\documentclass[12pt]{amsart}

\usepackage{amsmath}
\usepackage{amssymb}
\usepackage{amscd}
\usepackage{amsthm}
\usepackage[usenames,dvipsnames]{color}
\usepackage{epsfig,amsmath}
\usepackage{xypic}

\long\def\comment#1\endcomment{\blue{#1}}
\long\def\comment#1\endcomment{}

\def\a{\alpha}
\def\b{{\beta}}
\def\p{{\prime}}
\def \A{{\mathbb A}}
\def \I{{\text{i}}}
\def \bs{{\backslash}}
\def \CO{{\mathcal O}}
\def \C{{\mathcal C}}
\def \CC{{\mathbb C}}
\def \QQ{{\mathbb Q}}
\def \D{{\mathbb D}}
\def \M{{\mathcal M}}

\def \Z{{\mathcal Z}}
\def \K{{\mathcal K}}
\def \ZZ{{\mathbb Z}}

\def \PP{{\mathbb P}}

\def \HH{{\mathbb H}}

\def \RR{{\mathbb R}}
\def \Dd{{\mathcal D}}

\def\CH{\textup{CH}}
\def\cl{\textup{cl}}
\def\Tot{{\bf s}}
\def\del{{\partial}}
\def\ol{\overline}
\def\ul{\underline}
\def\Gy{{\text{Gy}}}
\def\MHS{{\text{MHS}}}

\def\Spec{\textup{Spec}}

\def\Ext{{\textup{Ext}}}
\def\coker{{\textup{coker}}}
\def\alg{\textup{alg}}
\def\dec{\textup{dec}}
\def\ind{\textup{ind}}
\def\et{\textup{et}}
\def\IM{\textup{im}}

\newtheorem{theorem}[equation]{Theorem}

\newtheorem{cor}[equation]{Corollary}

\newtheorem{prop}[equation]{Proposition}

\theoremstyle{definition}
\newtheorem{conj}[equation]{Conjecture}
\newtheorem{defn}[equation]{Definition}
\newtheorem{ex}[equation]{Example}
\newtheorem{rem}[equation]{Remark}

\numberwithin{equation}{section} 
\numberwithin{figure}{section}

\begin{document}

\title[Beilinson's Hodge Conjecture]{Beilinson's Hodge Conjecture for Smooth Varieties}

\author{Rob de Jeu}

\address{Faculteit Exacte Wetenschappen\\  Afdeling Wiskunde \\
VU University Amsterdam \\ The Netherlands}

\email{jeu@few.vu.nl}

\author{James D. Lewis}

\address{632 Central Academic Building\\
University of Alberta\\
Edmonton, Alberta T6G 2G1, CANADA}

\email{lewisjd@ualberta.ca}

\date{\today}

\thanks{The authors are grateful to the Netherlands
Organisation for Scientific Research (NWO) 
{for financially supporting a visit in November 2008 to VU University Amsterdam
by James Lewis.}
The second author also acknowledges partial {support} by a grant from the Natural Sciences
and Engineering Research Council of Canada.}



\renewcommand{\abstractname}{Abstract} 

\begin{abstract}  
Let $U/\CC$ be a smooth quasi-projective variety of dimension $d$, 
$\CH^{r}(U,m) $ Bloch's higher Chow group, and $\cl_{r,m} : 
\CH^{r}(U,m)\otimes \QQ \to 
\hom_{\MHS}\big(\QQ(0),H^{2r-m}(U,\QQ(r))\big)$
the cycle class map. Beilinson once conjectured $\cl_{r,m}$
to be surjective \cite{Be}; however Jannsen was the first to
find a counterexample in the case $m=1$ \cite{Ja1}. In this
paper we  study the image of $\cl_{r,m}$ in more
detail  (as well as at the ``generic point'' {of $ U$}) in terms of 
kernels of Abel-Jacobi mappings.
When $r=m$, we deduce from the Bloch-Kato conjecture
(now a theorem) various results, in particular that the cokernel of $\cl_{m,m}$ at the generic point 
is the same for integral or rational coefficients.
\end{abstract}

\maketitle

\tableofcontents{}

\section{\ Introduction}\label{SEC01}

Let $U/\CC$ be a smooth and quasi-projective variety of dimension $d$, 
$\CH^{r}(U,m) $ Bloch's higher Chow group, and 
$$
\cl_{r,m} : \CH^{r}(U,m; \QQ) \to 
\Gamma\big(H^{2r-m}(U,\QQ(r))\big), 
$$
the Betti cycle class map, where
\[
\Gamma\big(H^{2r-m}(U,\QQ(r))\big) := \hom_{\MHS}\big(\QQ(0),H^{2r-m}(U,\QQ(r))\big).
\]
If $m=0$, then by a standard weight argument (see \cite[pp.62-63]{Ja1}),
the Hodge conjecture implies that $\cl_{r,0}$ is surjective. Beilinson
once conjectured that $\cl_{r,m}$ is always surjective \cite{Be}. However
unless $U$ is given by base extension from a smooth quasi-projective
variety over a number field, it is known that this
conjecture is too optimistic \cite[Cor.~9.11]{Ja1}. 

\smallskip
\noindent
Consider these three statements:

\smallskip
\noindent
(S1) $\cl_{r,m} : \CH^{r}(X,m;\QQ) \to 
\Gamma\big(H^{2r-m}(X,\QQ(r))\big)$ is surjective for all smooth complex projective 
varieties $X$;

\smallskip
\noindent
(S2) $\cl_{r,m} : \CH^{r}(U,m;\QQ) \to
\Gamma\big(H^{2r-m}(U,\QQ(r))\big)$ is surjective for all smooth complex
quasi-projective varieties $U$;

\smallskip
\noindent
(S3) $\lim(\cl_{r,m}) : \CH^{r}(\Spec(\CC(X)),m;\QQ) \to
\Gamma\big(H^{2r-m}(\CC(X),\QQ(r))\big)$ is surjective for all smooth complex projective 
varieties $X$.

\smallskip

Note that (S1) for $ m=0 $ is equivalent with the 
Hodge conjecture (with rational coefficients 
as opposed to the original version with integral coefficients
that was disproven by Atiyah-Hirzebruch by showing the
existence of non-algebraic torsion classes),
and that it is trivially true for $m>0$ because
then $ \Gamma\big(H^{2r-m}(X,\QQ(r))\big) = 0$.
Also,
$\CH^{r}(\Spec(\CC(X)),m;\QQ) = 0$ for $r > m$ in (S3) because of dimension
reasons.

\smallskip
When $m=0$, all three statements (for all $r\geq 0$) are equivalent ({as one sees}
using a localization sequence argument, and Deligne's mixed Hodge theory,
as on \cite[pp.62-63]{Ja1}). 
However as we shall see in this paper, the statements are independent of each
other: (S1) is expected to be true, we conjecture
(S3) to be true, and show (S2) to be false in general.
There is some evidence (\cite{A-S}, \cite{Sa}, \cite{A-K}, \cite{MSa}) that (S2) holds
in the special case $r=m$, and the results in this paper
are consistent with this.
In \cite{SJK-L} we provided some 
evidence that~(S3) is always true, and in particular, 
(S3) can be viewed as an appropriate
generalization of the Hodge conjecture. 

\smallskip

In this paper we address a number of issues, namely:

\smallskip
\noindent
$\bullet$  Necessary and sufficient conditions for
$\cl_{r,m}$, and $\lim(\cl_{r,m})$ to be surjective, in terms
of kernels of (reduced) higher Abel-Jacobi maps.
This is worked out in Theorem \ref{T01} and subsequent examples,
as well as in Corollary \ref{C116} below.
Naturally this leads to a generalized notion of decomposable classes,
which is discussed in Section~\ref{SEC07}.
Also, in Theorem \ref{T45} we exhibit counterexamples to
the surjectivity of $\cl_{r,m}$ in (S2) in all cases where it
is not trivially true or
where one might reasonably expect this surjectivity
(namely $ r=m $ or $ m=0 $).

\smallskip
\noindent
$\bullet$ The story can be worked out  with integral coefficients in the case $r=m$, and
in particular, we are interested in the nature of the map
\begin{equation*}
d\log_m :\CH^m(\Spec(\CC(X)),m) \to H^m(\CC(X),\ZZ(m)) \cap F^mH^m(\CC(X),\CC).
\end{equation*}
We prove in Section~\ref{SEC010}
that the torsion subgroup $H^m(\CC(X),\ZZ(m))_{\rm tor}$
of $H^m(\CC(X),\ZZ(m))$ is trivial\footnote{Already known to some experts.}
(hence this intersection makes sense!).
The combination of Theorem \ref{T534} and Conjecture
\ref{CO534} below would imply that $d\log_m$ is surjective.
We also relate $d\log_m$ to the  map
\[
\frac{\CH^m(\Spec(\CC(X)),m)}{l} \to H^m_{\et}(\Spec(\CC(X)),\mu_l^{\otimes_{\ZZ}m}),
\]
for $l$ a non-zero integer,
which is now known to be an isomorphism. (This is the former
Bloch-Kato conjecture\footnote{This is now a theorem (\cite{W}).}, for the field $\CC(X)$.)
Thus the conjectured surjectivity of $d\log_m$ can be thought of as a Hodge
theoretic version of the Bloch-Kato conjecture.
Note that $ \lim(\cl_{m,m}) $ equals
\begin{equation*}
  d\log_m\otimes\ \QQ :
 \CH^m(\Spec(\CC(X)),m;\QQ) \to \Gamma\big(H^m(\CC(X),\QQ(m))\big)
\,.
\end{equation*}
As mentioned above, the classical Hodge conjecture, as originally formulated by
Hodge with integral coefficients, is false. But we wish to remind the
reader that it is false with integral coefficients even {\em modulo torsion}
(see \cite[p.67]{Lew}),
albeit expected (by optimists) to be true with  rational coefficients.
The following statement, again proven in Section~\ref{SEC010},
therefore seems rather remarkable.

\begin{theorem}\label{T534} 
$\coker(d\log_m) \simeq \coker( \lim(\cl_{m,m}))$. In particular,
$d\log_m$ is surjective $\Leftrightarrow \lim(\cl_{m,m})$ is surjective.
\end{theorem}

What this theorem tells us is that the Hodge theoretic analog of the Bloch-Kato
conjecture is the  surjectivity of $\lim(\cl_{m,m})$. Quite generally
we expect that the following is true.

\begin{conj}\label{CO534}
For all $ r,m\ge0 $, statement  (S3) holds.
\end{conj}
\noindent
By our earlier remarks this conjecture includes the  Hodge conjecture
and relates it to the (now proved) Bloch-Kato conjecture.

The authors wish to thank 
Spencer Bloch, H\'el\`ene Esnault, Marc Levine and Chuck Weibel
for useful conversations and/or correspondence.

\section{\ Notation}\label{SEC02}

(i) Unless otherwise specified, 
$X$ is a smooth complex projective variety of dimension $d$, and
$U$ is a smooth complex quasi-projective variety.

\smallskip
Let $\A\subseteq \RR$ be a subring.
\smallskip

(ii) $\A(r) := (2\pi\I)^r\A$ (Tate twist).

\smallskip

(iii) {If $ H $ is an $\A$-mixed Hodge structure (MHS), then we write
$\Gamma(H) := \hom_{\A-\MHS}(\A(0),H)$ and $J(H) := \Ext^{1}_{\A-\MHS}(\A(0),H)$}.

\smallskip

(iv) For a quasi-projective variety $V$,
$\CH^{r}(V,m)$ is the higher Chow group defined in \cite{B}, 
and $\CH^{r}(V) := \CH^{r}(V,0)$.

\smallskip

(v) $\CH^{r}(V,m;\QQ) := \CH^{r}(V,m)\otimes\QQ$.

\smallskip

(vi) We write

\[H^{2r-m}(\CC(X),\QQ(r)) = 
\lim_{{\buildrel \to\over {U\subset X}}}H^{2r-m}(U,\QQ(r)),
\]
the limit taken over all  Zariski open subsets of $X$.

(vii) We let $ H_{\Dd}^*(U,\A(r)) $ denote the (algebraic) Deligne-Beilinson
cohomology \cite{EV} of $ U $.

\section{\ Weight filtered spectral sequence}\label{SEC03}

We provide a breezy review of some of the ideas in~\cite[Section~3.1]{K-L}. 
We first recall the definition
of the higher Chow groups. Let $V/{\CC}$ be a 
quasi-projective variety. Put
$z^r(V) =$ free abelian group generated by subvarieties of
codimension $r$ in $V$, $\Delta^m$ the standard $m$-simplex, and
$z^r(V,m) = \big\{\xi\in z^k(V\times \Delta^m)\ \big|$ $\xi$ meets
all faces properly$\big\}$.  We let $ \partial = \sum_i (-1)^i \partial_i $
where $ \partial_i $ is the restriction to the $ i $-th codimension $1$ face.

\begin{defn}\label{DEFNA1} (\cite{B})
${\CH}^{\bullet}(V,\bullet) =$ homology of $\big\{z^{\bullet}
(V,\bullet),\partial\big\}$.
 We put ${\rm CH}^k(V) := {\rm CH}^k(V,0)$. 
 \end{defn}
 
We also need to recall the cubical version.
Let $\square^m:= ( \PP^1 \backslash
\{1\})^m $ with coordinates $z_i$ and $2^{m}$ codimension one faces
obtained by setting $z_i=0,\infty$, and boundary maps
$\del = \sum (-1)^{i-1}(\partial_{i}^{0}-\partial_{i}^{\infty})$,
where $\partial_{i}^{0},\ \partial_{i}^{\infty}$ denote
the restriction maps to the faces $z_{i}=0,\ z_{i}=\infty$ 
respectively. The rest of the definition is completely
analogous for $z^r(V,m) \subset z^r(V\times \square^m)$,
except that one has to quotient out by a
subgroup of degenerate cycles.
It is known that both complexes are quasi-isomorphic  (\cite{B}).

Now write
$U = X\bs Y$, where $X/\CC$ is a smooth projective variety of dimension $d$,
$Y = Y_{1}\cup\cdots\cup Y_{n}\subset X$ a NCD with smooth components.
For an integer $t\geq 0$,
put $Y^{[t]} =$ disjoint union of $t$-fold intersections
of the various components of $Y$, with corresponding
simplicial scheme $Y^{[\bullet]} \to Y \hookrightarrow
Y^{[0]} :=X$. 
There is a third quadrant double complex
\begin{equation} \label{ss}
\Z_{0}^{i,j}(r) := z^{r+i}(Y^{[-i]},-j), \
i,\ j \leq 0;
\begin{matrix} \Z_{0}^{i,j+1}\\
\quad\biggl\uparrow\del\\
\Z_{0}^{i,j}&{\buildrel {\Gy}\over\longrightarrow}&\Z_{0}^{i+1,j}
\end{matrix}
\end{equation}
whose differentials are $ \partial$ vertically
($ \partial $ as coming from the definition of Bloch's higher Chow groups),
and $ \Gy$ ($=$  Gysin) horizontally.
To the corresponding total complex $\Tot^{\bullet}\Z(r)$
with $D = \partial \pm \Gy$ are 
associated the {two} Grothendieck spectral sequences $ E_i^{p,q}$ and $ {}^{\p} E_i^{p,q} $
with
\[
E_{2}^{p,q} = H_{\Gy}^{p}\big(H_{\partial}^{q}(\Z_{0}^{\bullet,
\bullet}(r))\big);\quad
{}^{\p}E_{2}^{p,q} = H_{\partial}^{p}(H_{\Gy}^{q}\big(\Z_{0}^{\bullet,
\bullet}(r))\big).
\]
The second spectral sequence, together with
Bloch's quasi-isomorphism
$$
\frac{z^{\bullet}(X,\ast)}{z^{\bullet}_{Y}(X,\ast)}
{\buildrel {\text{\rm Restriction}}\over\longrightarrow} z^{\bullet}(U,\ast),
$$
shows that (\cite{K-L}(sect. 3.1))
$$
H^{-m}(\Tot^{\bullet}\Z(r)) = {}^{\p}E_{2}^{0,-m} = \CH^{r}(U,m).
$$
The first spectral sequence has $E_{1}^{i,j} = \CH^{r+i}(Y^{[-i]},-j)  $
and
\[
E_{2}^{i,j} = \frac{\ker\big( \Gy:\CH^{r+i}(Y^{[-i]},-j) \to
\CH^{r+i+1}(Y^{[-i-1]},-j)\big)}{\Gy\big(\CH^{r+i-1}(Y^{[-i+1]},-j)\big)}
\,.
\]
The corresponding filtration on $\Tot^{\bullet}\Z(r)$ also induces
a ``weight'' filtration
$$
W_{-m}\CH^{r}(U,m) \subseteq\cdots \subseteq W_{0}\CH^{r}(U,m) = \CH^{r}(U,m),
$$
which is characterized by the injection
\[
\partial_R^{\ell,m}: Gr_W^{\ell}\CH^r(U,m)  \ {=}\
E_{\infty}^{-\ell-m,\ell} \hookrightarrow\ 
\left\{\begin{matrix} \text{\rm A\ subquotient\ of}\\
\CH^{r-\ell-m}(Y^{[\ell+m]},-\ell)\end{matrix}\right\},
\]
for $ \ell = -m,...,0$, where $\partial_R^{\ell,m}$
is called  a residue map in \cite{K-L}.  It is easy to check that
\[
 E_{\infty}^{0,-m}
=
 W_{-m}\CH^r(U,m)\big) 
=
 {\rm Image}\big(\CH^r(X,m) \to \CH^r(U,m)\big) 
\,.
\]

\section{\ The image of the cycle class map}\label{SEC04}

Our main goal in this section is to prove Theorem \ref{T01}, which provides necessary
and sufficient conditions for the surjectivity of
$\cl_{r,m} : \CH^r(U,m;\QQ) \to \Gamma\big(H^{2r-m}(U,\QQ(r))\big)$.
The obstruction to surjectivity will be explained in terms of kernels
of Abel-Jacobi maps for the higher Chow groups.
We fix $ U $, $ r\ge0 $ and $ m\ge0 $.
Of particular interest is the top residue
$\partial_R(\xi) := \partial_R^{0,m}(\xi) \in E_{\infty}^{-m,0}$ for $\xi\in \CH^r(U,m)$, where
\begin{equation}\label{E666}
E_{\infty}^{-m,0} \subseteq E_2^{-m,0} = \frac{\ker\big(\Gy : \CH^{r-m}(Y^{[m]}) \to 
\CH^{r-m+1}(Y^{[m-1]})\big)}
{\Gy\big(\CH^{r-m-1}(Y^{[m+1]})\big)}.
\end{equation}
(In general, $E_2^{-m,0} \ne E^{-m,0}_{\infty}$.)
We use this to study the
cycle class map $\cl_{r,m}$  via the commutative diagram
\begin{equation}\label{E0001}
\begin{matrix}
\CH^r(U,m;\QQ)&\twoheadrightarrow&{E}_{\infty}^{-m,0}\otimes\QQ\\
&\\
\cl_{r,m}\biggl\downarrow\quad&&\biggr\downarrow\\
&\\
\Gamma\big(H^{2r-m}(U,\QQ(r))\big)&\hookrightarrow&\Gamma\big(Gr^{W}_{0}H^{2r-m}(U,\QQ(r))\big),
\,
\end{matrix}
\end{equation}
where the injectivity of the map on the bottom row follows from
the fact that $\Gamma\big(W_{-1}H^{2r-m}(U,\QQ(r))\big) = 0$.
(In general we use $ \QQ $-coefficients because
weight filtration in Hodge theory is defined for such coefficients. Exceptions to this situation
are discussed in sections \ref{SEC012} and \ref{SEC010}.)
With regard to the morphism
$$
\Gy : \CH^{r-m}(Y^{[m]}) \to \CH^{r-m+1}(Y^{[m-1]}),
$$
for $m\ge 1$, put
$$
\CH^{r-m}(Y^{[m]})^{\circ} := 
\Gy^{-1}\big(\CH_{\hom}^{r-m+1}(Y^{[m-1]})\big),
$$
where $\CH_{\hom}^{r-m+1}(Y^{[m-1]}) \subset \CH^{r-m+1}(Y^{[m-1]})$ is the 
subgroup of null-homologous cycles on $Y^{[m-1]}$.
Notice that for $m\geq 1$,
$$
E_{\infty}^{-m,0} = E_{m+1}^{-m,0} \subseteq E_2^{-m,0} \hookrightarrow
\frac{\CH^{r-m}(Y^{[m]})^{\circ}}{\Gy\big(
\CH^{r-m-1}(Y^{[m+1]})\big)},
$$
and for all $m\geq 0$, we have
$$
d_m : E_m^{-m,0} \to E_m^{0,-m+1},\quad E_{\infty}^{0,-m+1} = 
\frac{E_m^{0,-m+1}}{d_m(
E_m^{-m,0})}. 
$$
For $k=1,...,{m+1}$ put
\[
\widetilde{E}_k^{-m,0} =
\begin{cases} 
\frac{\CH^{r-m}(Y^{[m]};\QQ)^{\circ}}{\Gy\big(\CH^{r-m-1}(Y^{[m+1]};\QQ)\big)}&
\text{if $k=1$}\\
E_k^{-m,0}\otimes\QQ&\text{for $k=2,...,m+1$}
\end{cases}
\]
{so that $ \widetilde{E}_k^{-m,0} \subseteq \widetilde{E}_1^{-m,0}$
for $k \geq 1$. (Our main reason for introducing $\widetilde{E}_1^{-m,0}$ is so
that  the map $\b$ in diagram (\ref{E004}) below has a chance of
being surjective, as implied by the Hodge conjecture.)
For $ k=1,\dots,m $ we let}
$\widetilde{E}_k^{-m+k,-k+1} $ be
\[
\begin{cases}
\ker \big(\CH^{r-m+1}_{\hom}(Y^{[m-1]};\QQ) {\buildrel {\Gy}\over\longrightarrow}
\CH^{r-m+2}(Y^{[m-2]};\QQ)\big)&\text{if $k=1$,}\\
 E_k^{-m+k,-k+1}\otimes\QQ
&\text{for $k=2,...,m$,}
\end{cases}
\]
hence $ \widetilde{E}_k^{-m+k,-k+1} \subseteq E_k^{-m+k,-k+1} \otimes \QQ$
and $ d_k(\widetilde{E}_k^{-m,0}) \subseteq \widetilde{E}_k^{-m+k,-k+1} $.
Then~\eqref{E0001} becomes
\begin{equation}\label{E0002}
\begin{matrix}
\CH^r(U,m;\QQ)&\twoheadrightarrow&\widetilde{E}_{\infty}^{-m,0}\\
&\\
\cl_{r,m}\biggl\downarrow\quad&&\biggr\downarrow\\
&\\
\Gamma(H^{2r-m}(U,\QQ(r)))&\hookrightarrow&\Gamma(Gr^{W}_{0})
\end{matrix}
\end{equation}
where we abbreviate $ W_j H^{2r-m}(U,\QQ(r))$ to $ W_j $
and similarly for $ Gr_j^W $, and where
$
\widetilde{E}_{\infty}^{-m,0}\ {\buildrel {\rm def}\over =}\
\widetilde{E}_{m+1}^{-m,0}  =
 \bigcap_{k\geq 1}\ker(d_k:\widetilde{E}_k^{-m,0}\to\widetilde{E}_k^{-m+k,-k+1})
\subseteq \widetilde{E}_{1}^{-m,0}
$,
which equals $ E_\infty^{-m,0}\otimes\QQ $
because ${E}_{\infty}^{-m,0} = {E}_{m+1}^{-m,0}$.
Note that as in \cite[3.1]{K-L},
$$
\Gamma(Gr^W_0) =
\frac{\ker\left(\begin{matrix} \Gy : H^{r-m,r-m}(Y^{[m]},\QQ(r-m))\to\\
H^{r-m+1,r-m+1}(Y^{[m-1]},\QQ(r-m+1))\end{matrix}\right)}{\Gy
\big(H^{r-m-1,r-m-1}(Y^{[m+1]},
\QQ(r-m-1))\big)},
$$
and, for $0\leq k\leq m$, 
\begin{equation}\label{111}
Gr^W_{-k} = \frac{\ker\left(\begin{matrix} \Gy: H^{2r-2m+k}(Y^{[m-k]},\QQ(r-m+k)) 
\to\\
H^{2r-2m +k+2}(Y^{[m-k-1]},\QQ(r-m+k+1))\end{matrix}\right)}{\Gy
\big(H^{2r-2m+k-2}(Y^{[m-k+1]},
\QQ(r-m+k-1))\big)}.
\end{equation}

Using the differentials $d_1 := \Gy,d_2,\ldots,d_m$ we claim that there
is a commutative
diagram of exact sequences for  each $k = 1,\ldots,m$
\begin{equation}\label{E003}
\begin{matrix}
0&\to&\widetilde{E}^{-m,0}_{k+1}&\to&\widetilde{E}^{-m,0}_k&
{\buildrel {d_k}\over\longrightarrow}&\widetilde{E}^{-m+k,-k+1}_k\\
&\\
&&\a_{k+1}\downarrow\ \quad&&\a_k\downarrow \quad&&\lambda_k\downarrow\quad\\
&\\
0&\to&J(W_{-k-1})&\to&J(W_{-k})&
{\buildrel h_k\over {\twoheadrightarrow}}&
J(Gr^W_{-k})
\end{matrix}
\end{equation}
where $h_k$ is the obvious map,  $\lambda_k$ is the Abel-Jacobi map
(defined explicitly in Section~\ref{SEC05} below),  
and where the $\a_k$'s are characterized
as follows. If we assume $\a_k$ is defined,
then the definition of $\a_{k+1}$ is dictated by imposing commutativity in (\ref{E003}). Thus
we need only define $\a_1$, and show that  $h_k\circ \a_k = \lambda_k\circ d_k$.
The latter will be proven in Section~\ref{SEC05}.
Note that, as implicit in~\eqref{E0002},
$\cl_{r,m}$ is the composition
$$
\CH^r(U,m;\QQ) \to \widetilde{E}_{\infty}^{-m,0} = \widetilde{E}_{m+1}^{-m,0}
\to \Gamma\big(H^{2r-m}(U,\QQ(r))\big),
$$
and that there is a map
$ {\b=\b_{r,m}} : \widetilde{E}_1^{-m,0} \to  \Gamma(Gr^W_0)$,
which is an isomorphism when $r=m$, and is
surjective for all $r$ and $m$ under the assumption of the 
Hodge conjecture. We let $\a_1 =  \kappa \circ \b $ in
the diagram
(with $\kappa$ the obvious map)
\begin{equation}\label{E004}
\xymatrix{
\widetilde{E}_{\infty}^{-m,0} \,\, \ar[d]\ar@{^{(}->}[r]&\widetilde{E}_1^{-m,0}\ar[d]^-\b\ar[rd]^-{\a_1}\ar[rr]^-{d_1}&&\widetilde{E}_1^{-m+1,0}
\ar[d]^-{\lambda_1}
\\
\Gamma\big(H^{2r-m}(U,\QQ(r))\big) \,\, \ar@{^{(}->}[r]&\Gamma(Gr^W_0)
\ar[r]^-\kappa&J(W_{-1})\ar[r]^-{h_1} & J(Gr^W_{-1})
\,.
}
\end{equation}
Then~\eqref{E004} and~\eqref{E003} commute (see Theorem~\ref{diagrams-commute}).

With regard to the diagram~\eqref{E003} above, it
is obvious that 
$$
\ker (\a_{k+1}) = \ker\big(d_k\big|_{\ker(\a_k)}\big)\subseteq \ker(\a_k).
$$
From the isomorphisms
$\b(\ker(\a_k)) \simeq \ker(\a_k) / \ker(\b) \cap \ker(\a_k)
$ and $ \b(\ker(\a_{k+1})) \simeq \ker\big(d_k\big|_{\ker(\a_k)}\big)
/ \ker(\b) \cap \ker\big(d_k\big|_{\ker(\a_k)}\big)$,
we arrive at the identification
\begin{equation*}
 \frac{\b(\ker(\a_k))}{\b(\ker(\a_{k+1}))} \simeq 
\frac{d_k(\ker(\a_k))}{d_k\big(\ker(\b)\cap \ker (\a_k)\big)}
\,. 
\end{equation*}
We have inclusions
\[
 \IM(\cl_{r,m}) =
 \b(\ker(\a_{m+1})) \subseteq 
 \cdots \subseteq
 \b(\ker(\a_1)) \subseteq 
 \Gamma\big(H^{2r-m}(U,\QQ(r))\big)
\]
where on the left we have equality as $ W_{-m-1} =0 $
and $\widetilde{E}_{m+1}^{-m,0} = \widetilde{E}_{\infty}^{-m,0}$,
and the right-most inclusion is an equality if 
$\Gamma\big(H^{2r-m}(U,\QQ(r))\big)\subseteq\IM(\b)$ (e.g., if $\b$ is surjective).
We mention in passing the following result:

\begin{prop}\label{PP1} Suppose that $\lambda_k\big|_{\IM(d_k)}$ is injective
for $k=1,...,m$.  If $\b$ is surjective, then $\cl_{r,m} : \CH^r(U,m;\QQ)
\to \Gamma \big(H^{2r-m}(U,\QQ(r))\big)$ is surjective.
\end{prop}
\begin{proof} 
From ~\eqref{E003} we get $ \ker(\a_{m+1}) =\cdots=\ker(\a_1)$, and we apply the inclusions above.
\end{proof}

A slight tweaking of  Proposition~\ref{PP1} together with diagrams~\eqref{E003}
and~\eqref{E004} leads to:

\begin{prop} \label{P001}[M. Saito \cite{MSa}. Also see \cite{K-L}.]
Assume the Hodge conjecture, and that the Bloch-Beilinson conjecture holds, viz.,
for all smooth projective $V/\ol{\QQ}$,
the Abel-Jacobi map 
 $AJ : \CH_{\hom}^{r}(V/\ol{\QQ},j;\QQ) \to
J\big(H^{2r-j-1}(V,\QQ(r))\big)$, is injective for all $r$ and $j$.
If $U$ is obtained  from a smooth quasi-projective variety over $\ol{\QQ}$ by
base change to $\CC$, then for all $r$ and $m$,
$\cl_{r,m} : \CH^{r}(U,m;\QQ) \to
\Gamma\big(H^{2r-m}(U,\QQ(r))\big)$ is surjective.
\end{prop}

\begin{proof} The proof, which is similar to the one
given in  \cite[Prop.~3.7]{K-L}, is omitted.
\end{proof}

Next, $\ker(\b)\subseteq\ker(\a_1)$ by ~\eqref{E004}, hence 
$\ker(\b) \cap\widetilde{E}_k^{-m,0} \subseteq\ker(\a_1)\cap\widetilde{E}_k^{-m,0}
=\ker(\a_k)\subseteq \widetilde{E}_k^{-m,0}$ for $k=1,...,m+1$, with the last equality because of ~\eqref{E003}. Hence
$\ker(\b)\cap\ker(\a_k) =\ker(\b)\cap\widetilde{E}_k^{-m,0} $.
If $h_k$ is an isomorphism then this gives
\begin{equation*}
 \frac{d_k(\ker(\a_k))}{d_k\big(\ker(\b)\cap \ker (\a_k)\big)}
=
 \frac{\ker\big(\lambda_k\big|_{d_k(\widetilde{E}_k^{-m,0})}\big)}{d_k\big(\ker(\b)\cap\widetilde{E}^{-m,0}_k\big)}
\,.
\end{equation*}
This is the case when $k=m$ again because $ W_{-m-1}=0 $.
Putting all these ideas together, we obtain the following.

\begin{theorem}\label{T01} 

{\rm (i)} if $\cl_{r,m}$ is surjective then
$$
\frac{d_k(\ker(\a_k))}{d_k\big(\ker (\b) \cap \ker(\a_k)
\big)} = 0 \text{ for all } k=1,\ldots,m;
$$
{the converse is true if $ \b $ is surjective;}
{more precisely:}
{the converse is true if $ \Gamma\big(H^{2r-m}(U,\QQ(r)) \subseteq \IM(\b) $ (e.g., if $ \b $ is surjective);}

{\rm (ii)} $\cl_{r,m}$ is surjective implies that
$$
\frac{\ker\big(\lambda_m\big|_{d_m(\widetilde{E}_m^{-m,0})}\big)}{d_m\big(\ker (\b) \cap 
\widetilde{E}^{-m,0}_m\big)} = 0
$$

{\rm (iii)} if  $ \Gamma\big(H^{2r-m}(U,\QQ(r)) \subseteq \IM(\b) $ and
$$
\frac{d_k(\ker(\a_k))}{d_k\big(\ker (\b) \cap \ker(\a_k)\big)} = 0
$$
for all $k=1,\ldots,m-1$, then there is a short exact sequence
$$
0\to\IM(\cl_{r,m}) \to \Gamma\big(H^{2r-m}(U,\QQ(r))\big) \to
\frac{\ker\big(\lambda_m\big|_{d_m(\widetilde{E}_m^{-m,0})}\big)}{d_m\big(\ker (\b) \cap 
\widetilde{E}^{-m,0}_m\big)} \to 0
\,.
$$
\end{theorem}

Note that $\lambda_{1}$ is automatically injective when $ r=m $
by the theory of the Picard variety,  so that
$d_1(\ker (\a_1)) = 0$ in this case.  Since $\b$ is an isomorphism here,
we deduce

\begin{cor}\label{C56} Let us assume that $r=m$. Then
$\cl_{m,m}$ is surjective if and only if 
$d_{k}(\ker(\a_{k})) = 0$ for all $k=2,\ldots,m$. In particular,
$\cl_{1,1}$ is always surjective. (See also Example~\ref{EX1}.)
\end{cor}

\begin{ex}
Assume $r\geq m=1$. 
Note that $d_{1} : \widetilde{E}^{-1,0}_{1} \to \widetilde{E}^{0,0}_{1} = 
\CH_{\hom}^{r}(X;\QQ)$. Assuming $\b$ is surjective, 
then Theorem \ref{T01}(iii) we deduce the short exact 
sequence
$$
0 \to \IM(\cl_{r,1}) \to \Gamma\big(H^{2r-1}(U,\QQ(r))\big) \to
\frac{\ker\big(\lambda_{1}\big|_{\IM(d_{1})}\big)}{d_{1}(\ker(\b))}\to 0
\,.
$$
Recalling $U = X\bs Y$  of dimension $d$, we have
that
$$
\lambda_1 : \IM(d_1) \to 
J\big(H^{2r-1}(X,\QQ(r))/H_Y^{2r-1}(X,\QQ(r))\big),
$$
where the denominator term $H_Y^{2r-1}(X,\QQ(r))$ in the jacobian is identified with its
image in $H^{2r-1}(X,\QQ(r))$, which apparently coincides with
$\Gy\big(H^{2r-3}(Y^{[1]},\QQ(r-1))\big)$ by a standard mixed Hodge theory argument
(Deligne).
Taking limits, we arrive at the short exact sequence
$$
0\to \IM(\lim(\cl_{r,1})) \to \Gamma\big(H^{2r-1}(\CC(X),\QQ(r))\big)\to
$$
$$
 \frac{\ker \big(\CH^r_{\hom}(X;\QQ) {\buildrel AJ\over\longrightarrow} J\big(
H^{2r-1}(X,\QQ(r))/N^1H^{2r-1}(X,\QQ(r))\big)
\big)}{\lim_{\buildrel\to\over U}d_1(\ker(\b))}\to 0,
$$
where $N^pH^i(X,\QQ)$ is the $p$-th coniveau filtration. 
However, if for example $r=d$, then using the fact that a zero-cycle on
a projective variety is homologous to zero if and only its degree is $0$, we see
$\b$ is surjective in this case, and that
$$
\frac{\CH^d_{\hom}(X;\QQ)}{\displaystyle{\lim_{\buildrel\longrightarrow\over U}}
\, d_1(\ker(\b))} = 0,
$$
owing to the fact that any finite set of points on $X$ lies on a smooth divisor  in $X$.
Therefore we arrive at the following result.
\end{ex}

\begin{cor}  [\cite{SJK-L}]
$$
\lim(\cl_{d,1}) : \CH^d(\Spec(\CC(X)),1;\QQ) \to
\Gamma\big(H^{2d-1}(\CC(X),\QQ(d))\big)
$$
is surjective.
\end{cor}

\begin{ex}\label{EX01}  
The case $r=m=2$. We observe  that $\ker(\lambda_1) = 0$,
and that $\b$ is an isomorphism. Thus from
Theorem \ref{T01}(iii) and (\ref{E003}), we arrive at the
short exact sequence
$$
0 \to \IM(\cl_{2,2}) \to \Gamma\big(H^{2}(U,\QQ(2))\big) \to 
\ker\big(\lambda_{2}
\big|_{\IM(d_{2})}\big) \to 0.
$$
We have
$$
  d_2(\widetilde{E}_2^{-2,0}) \subseteq
\widetilde{E}_2^{0,-1} = \frac{\CH^2(X,1;\QQ)}{\Gy
\big(\CH^1(Y^{[1]},1;\QQ)\big)} \xrightarrow{\lambda_{2}}
J\biggl(\frac{H^2(X,\QQ(2))}{H^2_Y(X,\QQ(2))}\biggr).
$$
There is an exact sequence
$$
0\to d_2(\widetilde{E}_2^{-2,0}) \to 
\frac{\CH^2(X,1;\QQ)}{\Gy\big(\CH^1(Y^{[1]},1;\QQ)\big)}
\to \CH^2(U,1;\QQ),
$$
hence
$$
d_{2}(\widetilde{E}_2^{-2,0}) = \text{\rm Image\ of}\ 
\CH^{1}(Y,1;\QQ) \to \frac{\CH^2(X,1;\QQ)}{\Gy\big(\CH^1(Y^{[1]},1;\QQ)\big)}
\,.
$$
Note that 
$$
\CH^1(Y^{[1]},1) = \big(\CC^{\times}\big)^{\oplus n},
$$
and recall we have $Y^{[1]} = \coprod_{i=1}^n Y_i$. {\it Thus $\cl_{{2,2}}$ is surjective if
and only if
$\lambda_{2}$ is injective on the subgroup of cycles in 
$\CH^{2}(X,1;\QQ)$ supported on $Y$, modulo the image
of the space of decomposables in 
$\CH^{}(X,1;\QQ)$, supported on $Y$.} 
Now let $\CH_{\dec}^{2}(X,1;\QQ) := $
$$
\IM\big(\CH^{1}(X,0;\QQ) \otimes 
\CH^{1}(X,1;\QQ)
\to \CH^{2}(X,1;\QQ)\big),
$$
under the  product  on the higher Chow groups, and
$$
\CH^{2}_{\ind}(X,1;\QQ) := 
\frac{\CH^{2}(X,1;\QQ)}{\CH^{2}_{\dec}(X,1;\QQ)}.
$$
Then
$$
\lim(\cl_{2,2}): \CH^{2}\big(\Spec(\CC(X)),2;\QQ)\big) \to
\Gamma\big(H^{2}(\CC(X),\QQ(2))\big),
$$
is surjective if and only if
$$
AJ : \CH^{2}_{\ind}(X,1;\QQ) \to  J\biggl(
\frac{H^{2}(X,\QQ(2))}{N^{1}H^{2}(X,\QQ(2))}\biggr),
$$
is injective\footnote{This is also pointed out in  \cite{MSa} (Remark 4.4).}.
 In summary,

\begin{cor}
\[
\frac{\Gamma\big(H^{2}(\CC(X),\QQ(2))\big)}{\lim(\cl_{2,2})
\big(\CH^{2}\big(\Spec(\CC(X)),2;\QQ)\big)\big)} \simeq 
\]
\[
\ker AJ : \CH^{2}_{\ind}(X,1;\QQ) \to J\biggl(
\frac{H^{2}(X,\QQ(2))}{N^{1}H^{2}(X,\QQ(2))}\biggr).
\]
(See also Corollary \ref{C100}.)
\end{cor}

In particular, $\lim(\cl_{2,2})$ above
is surjective if $\CH^{2}_{\ind}(X,1;\QQ) = 0$. We recall Bloch's
conjecture which says in the case that $X$ is a surface,
$p_g(X) = 0 \Leftrightarrow$ the Albanese map $\kappa: \CH^2_{\deg 0}(X)
\to$ Alb$(X)$ is an isomorphism. Equivalently, this amounts to saying 
that $p_g(X) = 0 \Leftrightarrow$ the motive of $X$ degenerates.
The degeneration of the motive of $X$ implies that 
$\CH^{2}_{\ind}(X,1;\QQ) = 0$ (\cite{CS}).
So according to Bloch's
conjecture, if $X$ is a surface with $p_g(X) = 0$, then $\lim(\cl_{2,2})$  
is surjective.
\end{ex}

\section{\ Amending the Beilinson-Hodge conjecture}\label{SEC06}

For any smooth quasi-projective variety $U/\CC$ of dimension~$ d $
we consider the following three regions for the pair $(r,m)$
(see Figure~\ref{regions}):
\begin{enumerate}
\item[I:]
$r>m>0$ and $ r\le d $;

\item[II:]
$r>m$ and $r>d$;

\item[III:]
$r<m$.
\end{enumerate}

\begin{figure}[htbp]
\begin{center}
\input{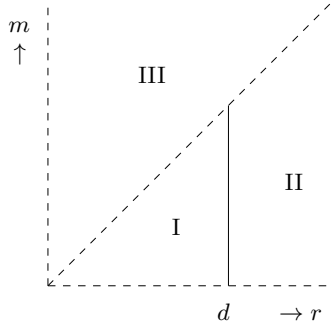}
\caption{The three regions}
\label{regions}
\end{center}
\end{figure}
The corresponding cycle class map 
\begin{equation*}
\cl_{r,m}: \CH^{r}(U,m;\QQ) \to \Gamma\big(H^{2r-m}(U,\QQ(r))\big)
\end{equation*}
is surjective in regions II and III
since there the right-hand side is trivial
(see Corollary 6.6 on page 85 in \cite{Ja1}).
We shall show below that
 at every point in region~I surjectivity fails in general.
Thus the only open cases are the diagonal ($ r=m $),
where surjectivity is the Beilinson-Hodge conjecture
as formulated in \cite{A-S}, 
and the $r$-axis, where surjectivity corresponds to the  Hodge conjecture
extended to such $U$.

\begin{theorem}\label{T45} {For a fixed $ d $, assume} $(r,m)$
lies in region~I.
Then there exists a smooth quasi-projective variety $U/\CC$
of dimension $d$ such that $\cl_{r,m}$ fails to be
surjective.
\end{theorem}

\begin{proof}
Special instances of this are already
established in \cite{K-L}. Let $Z_{0}\subset \PP^{r+1}$ be a hypersurface of
sufficiently large degree,
$\PP^{r}_{1},\ldots,\PP^{r}_{m-1}\subset \PP^{r+1}$ general
hyperplanes such that if we put
$$
W = \PP^{r}_{1}\cap\cdots\cap\PP^{r}_{m-1}\cap Z_{0},
$$
then $ W $ is smooth and $H^{0}(W,\Omega_{W}^{q}) \ne 0$ where
$q = r-m+1 = \dim W \geq 2$ since $r>m$.
Fix points $P,\ Q$ in $W$ such that the class of $P-Q$ is non-trivial in $\CH_{0}(W;\QQ)$
(possible by Mumford$/$Roitman, see \cite[Ch.15]{Lew}),
and consider the blow-up
$$
Z := B_{\{P,Q\}}(Z_{0}) {\buildrel \pi\over\longrightarrow} Z_{0}.
$$
Set 
$ E_{1} = \pi^{-1}(P) $, $ E_{2} = \pi^{-1}(Q) $,
$ E_{2+i} = \ol{\pi^{-1}\big( \PP^{r}_{i}\cap Z_{0} \bs \{P,\ Q\}\big)} $ for 
$ i=1,\ldots,m-1 $, and $ E = \bigcup_{j=1}^{m+1}E_j $.
Observe that
\[
\bigcap_{i=1}^{m-1} E_{2+i} = B_{\{P,Q\}}(W).
\]
Finally for $k  = d-r\geq 0$,
put
\[
X = Z\times \PP^{k}
\]
\[
Y = \bigcup_{i=1}^{m+1}E_{i}\times \PP^{k}
\]
so that $ X $ has dimension~$ d\ge2 $.  Note that 
\[
Y^{[m+1]} = \emptyset,
\]
\[
Y^{[m]} = \biggl\{\big\{E_1 \cap B_{\{P,Q\}}(W)\big\}\times \PP^k\biggr\} \ \coprod\ 
\biggl\{\big\{E_2\cap B_{\{P,Q\}}(W)\big\}
\times \PP^k\biggr\}
\]
\[
 \simeq \big\{\PP^{r-m}\times \PP^k\big\} \coprod \big\{\PP^{r-m} \times \PP^k\big\},
\]
and that
$B_{\{P,Q\}}(W) \times \PP^k$ is an irreducible component
of $Y^{[m-1]}$.
Then with regard to  diagram~\eqref{E004}, $\b$
is an isomorphism, and yet for $k = d-r$ there is a class
$$
\xi \in \CH^{r-m}(Y^{[m]};\QQ)^{\circ}
$$
of the form $\{0-$cycle$\}\times\PP^k$, for which 
$$
\Gy(\xi) \ne 0 \in \CH^{r-m+1}_{\hom}(Y^{[m-1]};\QQ),
$$
but $\a_{1}(\xi) = 0 \in J(W_{-1})$.  To see why $\a_{1}(\xi) = 0$, observe
that since $W_{-1}$ has negative weight, it suffices to show that the values
of $\a_{1}(\xi)$ map to zero in $J(Gr_{-k}^W)$ for $k\geq 1$. 
But the relevant part of $Gr_{-k}^W$ involves the cohomology
$H^{2r-2m+k}(E^{[m-k]})\otimes H^0(\PP^k) \subset  H^{2r-2m+k}(Y^{[m-k]})$,
which in the end involves the mixed Hodge structure of
\[
Z\bs E= Z_0 \bs 
\underbrace{\{\PP^r_1\cap Z_0 \cup\cdots\cup \PP^r_{m-1}\cap Z_0\}}_{=: E_0}.
\]
But $E_0^{[m-k]}$ is a 
union of smooth hypersurfaces of dimension $r-m+k$. Since by Lefschetz, 
the cohomology of hypersurfaces is only ``non-trivial'' in the middle dimension,
and in light of the description of  $Gr_{-k}^W$ in (\ref{111}), it suffices to show
that $2r-2m+k \ne r-m+k$ (hence $Gr_{-k}^W = 0$, as it reduces
to the same thing as the homology of a simplex).
But $2r-2m+k = r-m+k \Leftrightarrow r=m$ which is not the case for region~I.
Thus $\a_{1}(\xi) = 0$, 
hence  $\lambda_{1}\circ d_{1}(\xi) = 0$ 
as well. In particular,
$\xi\in \ker(\a_1)$, $\ker(\b) = 0$, and $d_{1}(\xi) = 
\Gy(\xi) \ne 0$. 
Thus by Theorem \ref{T01}(i), $\cl_{r,m}$ fails to be surjective. 
\end{proof}

\section{\ Integral coefficients I}\label{SEC012}

This section serves as a necessary forerunner to Section~\ref{SEC010}.
Along the way we prove some results that are either new, or appear to be  known
only among experts.
Let $X/\CC$ be a smooth projective variety and $Y\subset X$ a proper subvariety. There is
a short exact sequence
\[
0\to \frac{H^{2r-m}(X,\ZZ(r))}{H_Y^{2r-m}(X,\ZZ(r))} \to H^{2r-m}(X\bs Y,\ZZ(r))
\to H^{2r-m+1}_Y(X,\ZZ(r))^0\to 0,
\]
where, for notational simplicity,
we write $H_Y^{2r-m}(X,\ZZ(r))$ instead of $\IM(H_Y^{2r-m}(X,\ZZ(r)))$,
and let
\[
H^{2r-m+1}_Y(X,\ZZ(r))^0 := \ker\big(H^{2r-m+1}_Y(X,\ZZ(r)) \to H^{2r-m+1}(X,\ZZ(r))\big).
\]
Let us assume for the moment that
\[
\frac{H^{2r-m}(X,\ZZ(r))}{H_Y^{2r-m}(X,\ZZ(r))}
\]
is torsion-free.
Except for the obvious case $ 2r-m = 0$ this also holds in the
following two cases:

\smallskip
(i) $2r-m=1$. Here
\[
\frac{H^{1}(X,\ZZ(r))}{H_Y^{1}(X,\ZZ(r))} = H^1(X,\ZZ(r)),
\]
is torsion-free, as can be seen from the long exact sequence
of cohomology of $X$ associated to the short exact sequence
\begin{equation} \label{expsequence}
0 \to \ZZ(1) \to \CC \to \CC^{\times} \to 0.
\end{equation}

(ii) $2r-m=2$. Let  $Y$ be a divisor such that the image 
$H^2_Y(X,\ZZ(2)) \to H^2(X,\ZZ(2))$ is precisely the algebraic
part $H^2_{\alg}(X,\ZZ(2))$. Then by
the Lefschetz $(1,1)$ theorem,
$H^{2}(X,\ZZ(2))\big/H^{2}_{\alg}(X,\ZZ(2))$
is torsion-free, 
{and $ H^{2}(X,\ZZ(2)) / H_Y^{2}(X,\ZZ(2)) $ is isomorphic to this group.}

\bigskip
Then by purity of negative weight and torsion-freeness, 
\[
\Gamma\biggl(\frac{H^{2r-m}(X,\ZZ(r))}{H_Y^{2r-m}(X,\ZZ(r))}\biggr) = 0 \ {\rm for}\ m>0.
\]
Corresponding to this is a commutative
diagram (use the fact that cycle class maps are compatible with localization sequences for
the left-hand square, and the commutativity of the right-hand square can be deduced from an extension class interpretation of the Abel-Jacobi map (see \cite{KLM}))
\begin{equation}\label{E27}
\begin{matrix}
\frac{\CH^r(X\bs Y,m)}{\CH^r(X,m)}&\hookrightarrow&\CH^r_Y(X,m-1)^0&\xrightarrow{\alpha}&\CH_{\hom}^r(X,m-1)\\
&\\
\ul{\cl}_{r,m}\biggl\downarrow\quad&&\ \biggl\downarrow\ul{\b}&&\quad\biggr\downarrow \ul{AJ}\\
&\\
\Gamma H^{2r-m}(X\bs Y,\ZZ(r))&\hookrightarrow&\Gamma H^{2r-m+1}_Y(X,\ZZ(r))^0&\to&
J\biggl(\frac{H^{2r-m}(X,\ZZ(r))}{H_Y^{2r-m}(X,\ZZ(r))}\biggr),
\end{matrix}
\end{equation}
where
\[
\CH^r_{\hom}(X,m-1) = \ker\big(\CH^r(X,m-1) \to H^{2r-m+1}(X,\ZZ(r))\big),
\]
and $\CH^r_Y(X,m-1)^0 :=$
\[
 \ker\big(\CH^r_Y(X,m-1) \to H^{2r-m+1}(X,\ZZ(r))\big) =
\alpha^{-1}\CH^r_{\hom}(X,m-1).
\]

\begin{ex}\label{EX1} 
Suppose that $r=m=1$. Then $\ul{\b}$ and $\ul{AJ}$ are isomorphisms,
{hence the same holds for} $\ul{\cl}_{1,1}$ by~\eqref{E27}.
If we make the identifications
$\CH^1(X,1) \simeq \CC^{\times}$, $\CH^1(X\bs Y,1) = \CO^{\times}_{X\bs Y}(X\bs Y)$, then
we arrive at the short exact sequence
\[
0 \to \CC^{\times}\to  \CO^{\times}_{X\bs Y}(X\bs Y) \xrightarrow{d\log} \Gamma H^1(X\bs Y,\ZZ(1))
\to 0,
\]
where ${\cl}_{1,1} = d\log $ is well-known (see  \cite{KLM}).
This is also a consequence of the identification
\[
\CO^{\times}_{X\bs Y}(X\bs Y)  \simeq H^1_{\Dd}(X\bs Y,\ZZ(1))
\,.
\]
The surjectivity of $d\log$ in this case  is also proven in \cite{A-K}.
\end{ex}

\begin{ex}\label{EX2}
Suppose that $(r,m) = (2,1)$, and $Y = Y_1\cup\cdots \cup Y_n\subset X$ is
a divisor.  We observe that there is a short exact sequence
\[
\frac{H^2_Y(X,\CC)}{F^2H^2_Y(X,\CC) + H^2_Y(X,\ZZ(2))}\hookrightarrow
H^3_{\Dd,Y}(X,\ZZ(2)) \twoheadrightarrow \Gamma H^3_{Y}(X,\ZZ(2)),
\]
where the first term may be identified with
\[
   \oplus_{j=1}^n H^0(Y_j,\CC/\ZZ(1))
\simeq
   (\CC^\times)^{\oplus n}
=:
   \CH^2_{Y,\dec}(X,1)
.
\]
There is a canonical isomorphism
$ \CH^2_Y(X,1) \xrightarrow{\sim} H^3_{\Dd,Y}(X,\ZZ(2)) $
by \cite[Lemma~3.1]{Ja2}.
(In loc.\ cit.\ this is formulated in terms of $K$-theory,
but because of the particular indices, it gives exactly
this result.) 
We therefore obtain a short exact sequence
(using the fact that $ \ul{\b} $ factors through Deligne cohomology)
\[
0 \to \CH^2_{Y,\dec}(X,1) \to \CH^2_{Y}(X,1)^0\xrightarrow{\ul{\b}}
\Gamma H^3_{Y}(X,\ZZ(2))^0 \to 0,
\]
so that the snake lemma applied to~\eqref{E27} yields an exact
sequence
\[
0\to \ker(\ul{\cl}_{2,2}) \to  \CH^2_{Y,\dec}(X,1) \to \ker\big(\ul{AJ}\big|_{{\rm Im}(\alpha)}\big)
\to {\rm cok}(\ul{\cl}_{2,2}) \to 0.
\]
\end{ex}

When taking limits over $ Y $ in the above example $ \alpha $ becomes surjective,
and $ \lim(\ul{\cl}_{2,2}) = d\log_2 $ in Section~\ref{SEC01}.
So using the description of $ \ker(\ul{\b}) $ for each $ Y $ above we obtain
the following result.

\begin{cor}\label{C100}
\[
\frac{\Gamma H^2(\CC(X),\ZZ(2))}{\IM(d\log_2)} \simeq
\ker\biggl[\frac{\CH_{\hom}^2(X,1)}{\CH^2_{\dec}(X,1)} \xrightarrow{\ul{AJ}} J\biggl(
\frac{H^2(X,\ZZ(2))}{H^2_{\alg}(X,\ZZ(2))}\biggr)\biggr].
\]
\end{cor}

Let us now consider the general situation of $(r,m)$
{with $ m>0 $}, so that we have the diagram~\eqref{E27}
tensored with $\QQ$. Then $\ul{\b}\otimes\QQ$ need
not be surjective; moreover  a detailed description of this map when
 $Y$ is a NCD leads
to the same kind of analysis as in Section~\ref{SEC04}.
Recall that by a weight argument,
$\Gamma H^{2r-m}(X\bs Y,\QQ(r)) = 0$ for $r \leq m-1$. 
With this in mind let us assume that $r\geq m$.
Then as $Y\subset X$ ranges
over all  pure codimension one subvarieties, the image of $\alpha$ in
(\ref{E27}) is
$\CH^r_{\hom}(X,m-1)$ because of dimensions. Referring to~\eqref{E27}$ {}_\QQ $, let us  put
\begin{equation}\label{E88}
N^1\CH^r(X,m-1;\QQ) := \lim_{{\buildrel \longrightarrow\over Y}}\alpha(\ker \ul{\b}\otimes\QQ),
\end{equation}
where $Y\subset X$ ranges over all pure  codimension 
one\footnote{Working with subvarieties of higher pure codimension, 
this gives rise to a descending filtration $\{N^p\CH^i(X,j;\QQ)\}_{p\geq 0}$ which
is finer that the coniveau filtration on $\CH^i(X,j;\QQ)$.}  algebraic subsets of $X$.
Note that from Example \ref{EX2}, $N^1\CH^2(X,1;\QQ) = \CH^2_{\dec}(X,1;\QQ)$.
Now fix a $j:Y\hookrightarrow X$  of pure codimension one, with desingularization
$\lambda : \widetilde{Y} \xrightarrow{\approx}Y$, and composite 
morphism $\sigma = j\circ\lambda: \widetilde{Y} \to X$.
By a weight argument (Deligne) together with the purity of $H^{2r-m}(X,\QQ(r))$,
 both Gysin images ($\sigma_{\ast}, j_{\ast}$)
in the commutative diagram below are
the same.
\[
\xymatrix{
H^{2r-m-2}(\widetilde{Y},\QQ(r-1)) \ar[rd]^-{\sigma_{\ast}} \ar[dd]_-{\lambda_{\ast}}
\\
&H^{2r-m}(X,\QQ(r))
\\
H_Y^{2r-m}(X,\QQ(r))\ar[ru]_-{j_{\ast}}
\,.
}
\]
Assuming the Hodge conjecture, we can find
$w$ in $ \CH^{d-1}(X\times\widetilde{Y};\QQ)$ with
$\sigma_{\ast}\circ [w]_{\ast} = {\rm Id}_{\IM(\sigma_{\ast})}$,
where 
\begin{equation*}
[w]_{\ast} : H^{2r-m}(X,\QQ(r)) \to H^{2r-m-2}(\widetilde{Y},\QQ(r-1))
\end{equation*}
is induced by $w$ (see~\cite[Prop.~7.4]{Lew}).
Note that 
\begin{equation}\label{EE8}
w_{\ast}\CH^r_{\hom}(X,m-1;\QQ) \subset \CH_{\hom}^{r-1}(\widetilde{Y},m-1;\QQ) \to
\ker\ul{\b}_{\QQ},
\end{equation}
where $\ul{\b}_{\QQ}  := \ul{\b}\otimes\QQ$. Let us similarly write
 \eqref{E27}$_{\QQ}$ for~\eqref{E27}
tensored with $ \QQ $, and let
\[
AJ_{\QQ} : \CH^r_{\hom}(X,m-1;\QQ) \to J\big(H^{2r-m}(X,\QQ(r))\big)
\]
be the (full) Abel-Jacobi map.  
Referring to  (\ref{E27}$)_{\QQ}$, we deduce
from the Hodge conjecture that
\[
\ker \big(\ul{AJ}_{\QQ}\big|_{\IM (\alpha)} \big) = \ker \big(AJ_{\QQ}\big|_{\IM (\alpha)}\big) + 
\alpha(\ker\ul{\b}_{\QQ}).
\]
Indeed, if $\xi \in \ker \big(\ul{AJ}_{\QQ}\big|_{\IM (\alpha)} \big)$, then 
$AJ_{\QQ}(\xi) = AJ_{\QQ}(\sigma_{\ast}\circ w_{\ast}(\xi))$ by functoriality of the Abel-Jacobi map.
Thus $\xi = \big(\xi - \sigma_{\ast}\circ w_{\ast}(\xi)\big) + \sigma_{\ast}\circ w_{\ast}(\xi)
\in  \ker \big(AJ_{\QQ}\big|_{\IM (\alpha)}\big) + \alpha(\ker\ul{\b}_{\QQ})$
by~\eqref{E27}${}_{\QQ}$.
In particular,
\begin{prop} Under the assumption of the
Hodge conjecture,
for a fixed $ Y $ as above
there are short exact sequences
\[
 \frac{ \ker AJ_{\QQ} \big|_{\IM (\alpha)}+ \alpha(\ker\ul{\b}_{\QQ})}{\alpha(\ker\ul{\b}_{\QQ})}\hookrightarrow
 \frac{\Gamma H^{2r-m}(X\bs Y,\QQ(r))}{{\IM}(\cl_{r,m})} 
 \twoheadrightarrow {\rm cok}(\ul{\b}_{\QQ})^0,
 \]
 where 
 \[
 {\rm cok}(\ul{\b}_{\QQ})^0 := \ker\big({\rm cok}(\ul{\b}_{\QQ}) \to {\rm cok}(\ul{AJ}_{\QQ})\big),
 \]
 \[ 
 \ul{AJ}_{\QQ} : {\IM}(\alpha) \to J\biggl(\frac{H^{2r-m}(X,\QQ(r))}{H_Y^{2r-m}(X,\QQ(r))}\biggr)
\,.
 \]
Taking the direct limit over $ Y $ we obtain a short exact sequence
 \[
 \frac{ \ker AJ_{\QQ} + N^1\CH^r(X,m-1;\QQ)}{N^1\CH^r(X,m-1;\QQ)}\hookrightarrow
 \frac{\Gamma H^{2r-m}(\CC(X),\QQ(r))}{{\IM}(\lim\cl_{r,m})} 
  \twoheadrightarrow \lim{\rm cok}(\ul{\b}_{\QQ})^0,
 \]
 where
 \[
  \lim{\rm cok}(\ul{\b}_{\QQ})^0 = \ker\big( \lim{\rm cok}(\ul{\b}_{\QQ}) \to {\rm cok}(\ul{AJ}_{\QQ})\big),
  \]
  \[ 
 \ul{AJ}_{\QQ} : \CH^r_{\hom}(X,m-1;\QQ) \to   J\biggl(\frac{H^{2r-m}(X,\QQ(r))}{N^1H^{2r-m}
 (X,\QQ(r))}\biggr).
 \]
\end{prop}

\begin{cor}\label{C116} Assume the Hodge conjecture and let $r\geq m$. Then
\[
 \frac{\Gamma H^{2r-m}(\CC(X),\QQ(r))}{{\IM}(\lim\cl_{r,m})}  = 0
\text{ implies }
 \ker AJ_{\QQ} \subset N^1\CH^r(X,m-1;\QQ).
 \]
\end{cor}

\begin{rem} 
(i)  Note that Corollary~\ref{C116} for $m=1$
is essentially a conjectural type question of Jannsen \cite[p.227]{Ja3}.

\smallskip
\noindent
(ii)  Let us (again) take the direct limit over $Y$ of diagram (\ref{E27}$)_{\QQ}$. 
By applying the snake lemma to the limit diagram, we deduce (unconditionally)
that for $r\geq m$, 
\[
\frac{\Gamma H^{2r-m}(\CC(X),\QQ(r))}{{\IM}(\lim\cl_{r,m})}  = 0
\]
implies that 
\[
\ul{AJ}_{\QQ} : \frac{\CH^r(X,m-1;\QQ)}{N^1\CH^r(X,m-1;\QQ)} \to
J\biggl(\frac{H^{2r-m}(X,\QQ(r))}{N^1H^{2r-m}(X,m-1;\QQ)}\biggr),
\]
is injective, which in turn implies by a
generalization of Beilinson rigidity theorem given in \cite{MS}
that
\[
 \frac{\CH^r(X,m-1;\QQ)}{N^1\CH^r(X,m-1;\QQ)}
 \]
is countable for $m\geq 2$. Note that in the case $r=m=2$, we have
\[
\CH^2_{\ind}(X,1;\QQ) =  \frac{\CH^2(X,1;\QQ)}{N^1\CH^2(X,1;\QQ)},
\]
where $\CH^2_{\ind}(X,1;\QQ)$ was defined in
Example \ref{EX01},
and the statement of countability of $\CH^2_{\ind}(X,1;\QQ)$ is
a conjecture of Voisin.
\end{rem}

\section{\ Integral coefficients II}\label{SEC010}

As always, $ X/\CC $ is a smooth projective variety.
This section concerns the following integrally defined  map:

\begin{equation}\label{E001}
\begin{aligned}
d\log_m : \CH^{m}(\Spec(\CC(X)),m) &\to \Gamma\big(H^{m}(\CC(X),\ZZ(m))\big),
\\
\{f_1,...,f_m\} &\mapsto \bigwedge_{1}^md\log(f_j)
\end{aligned}
\end{equation}
mentioned in Section~\ref{SEC01}.
Of course, $d\log_m \otimes \QQ = \lim(\cl_{m,m})$.
We shall prove in this section that $H^{m}(\CC(X),\ZZ(m))$ is torsion-free,
so that by a weight argument
\[
 \Gamma\big(H^{m}(\CC(X),\ZZ(m))\big)
=
 H^m(\CC(X),\ZZ(m)) \cap F^mH^m(\CC(X),\CC).
\]
Clearly, if $ d\log_m $ is surjective then so is 
$ \lim(\cl_{m,m})$, but we shall show that the converse also
holds.
In fact, we expect the following to be true:

\begin{conj}\label{C0001}
The map in (\ref{E001}) is surjective.
\end{conj}

For the moment,
let us restrict to the case $r=m=1$. It is then easy (and also follows from Example \ref{EX1})
that
 \[
 d\log : \CH^1(\CC(X),1) = \CC(X)^{\times} \to
  \Gamma\big(H^{1}(\CC(X),\ZZ(1))\big),
  \]
is surjective, with divisible kernel: $\ker ( d\log) = \CC^{\times}$.
Thus for any integer $l\ne 0$,
\begin{equation*}
\frac{\CH^1(\CC(X),1)}{l\cdot \CH^1(\CC(X),1)} \simeq 
\frac{ \Gamma\big(H^{1}(\CC(X),\ZZ(1))\big)}{l\cdot  \Gamma\big(H^{1}(\CC(X),\ZZ(1))\big)}
\,.
\end{equation*}
If $ U $ is a Zariski open part of $ X $, then
there is an exact sequence
\begin{equation}\label{E111}
0\to\frac{H^{1}(U,\ZZ(1))}{\Gamma\big(H^{1}(U,\ZZ(1))\big)}
\to \frac{H^{1}(U,\CC)}{F^{1}H^{1}(U,\CC)}
\to H_{\Dd}^{2}(U,\ZZ(1))
\,,
\end{equation}
where
\begin{equation}\label{E113}
H^{2}_{\Dd}(U,\ZZ(1)) = \CH^{1}(U)
\,.
\end{equation}
Indeed, from mixed Hodge theory (see \cite{De}(Cor. 3.2.13(ii)),
 the restriction map induces an isomorphism
\begin{equation*}
 H^{1}({X},{\CO}_{X})
\simeq
 \frac{H^{1}(U,\CC)}{F^{1}H^{1}(U,\CC)}
\,.
\end{equation*}
This together with the surjectivity of $\CH^1(X)\to \CH^1(U)$, 
implies that the restriction map $H_{\Dd}^2(X,\ZZ(1)) \to H_{\Dd}^2(U,\ZZ(1))$
is surjective; moreover by a weak purity argument, $H_{\Dd,X\bs U}^2(X,\ZZ(1)) \simeq
H^2_{X\bs U}(X,\ZZ(1))$. Via the Deligne cycle class maps,
we have identifications, $\CH^1(X) = H_{\Dd}^2(X,\ZZ(1))$,
$\CH^1_{X\bs U}(X) = H^2_{X\bs U}(X,\ZZ(1))$, and hence by a localization sequence
argument, the aforementioned identification $H^{2}_{\Dd}(U,\ZZ(1)) = \CH^{1}(U)$.
\comment
Now assume given a ``good compactification'' pair $(X,U)$, with NCD $Y := X\bs U$.
There is an exact sequence:
\begin{equation}\label{E111}
0\to\frac{H^{1}(U,\ZZ(1))}{\Gamma\big(H^{1}(U,\ZZ(1))\big)}
\to \frac{H^{1}(U,\CC)}{F^{1}H^{1}(U,\CC)}
\to H_{\Dd}^{2}(U,\ZZ(1)).
\end{equation}
Further,
\begin{equation*}
 H^{1}({X},{\CO}_{X})
\simeq
 \HH^{1}(\Omega^{\bullet < 1}_{{X}}\langle Y\rangle)
\simeq
 \frac{H^{1}(U,\CC)}{F^{1}H^{1}(U,\CC)}
\,,
\end{equation*}
and
\begin{equation}\label{E113}
H_{\Dd}^2(U,\ZZ(1)) = \CH^{1}(U).
\end{equation}
\endcomment
Next, by shrinking $U$, and using that $\CH^1(\Spec(\CC(X)) = 0$,
we obtain from~\eqref{E111} and~\eqref{E113} a short exact sequence
\begin{equation*}
 0 \to \Gamma H^1(\CC(X),\ZZ(1)) \to H^1(\CC(X),\ZZ(1)) \to  H^1(X,\CO_{X}) \to 0 
\,,
\end{equation*}
where the last term is uniquely divisible.  Hence, for $ l\ne0 $,
we find an isomorphism
\begin{equation*}
\frac{\Gamma H^{1}(\CC(X),\ZZ(1))}{l\cdot \Gamma H^{1}(\CC(X),\ZZ(1))}
\simeq
\frac{H^{1}(\CC(X),\ZZ(1))}{l\cdot H^{1}(\CC(X),\ZZ(1))}.
\end{equation*}

Next, we observe that $ H^2(\CC(X),\ZZ(1)) $ is torsion-free,
since by the Lefschetz $(1,1)$ theorem
the torsion in $H^2({X},\ZZ(1))$ is algebraic.
This implies that
\[
\frac{H^{1}(\CC(X),\ZZ(1))}{l\cdot H^{1}(\CC(X),\ZZ(1))} \simeq 
H^1\biggl(\CC(X),\frac{\ZZ(1))}{l\cdot \ZZ(1)}\biggr).
\]
Finally, we deduce the well-known fact
\[
\frac{\CH^1(\CC(X),1)}{l\cdot \CH^1(\CC(X),1)} \simeq
H^1\biggl(\CC(X),\frac{\ZZ(1))}{l\cdot \ZZ(1)}\biggr).
\]
To see this in another context, let us work in the \'etale topology on a 
variety $V/\CC$,
and consider the sheaf $\mu_l$ on $V$, where
for $U \to V$ \'etale, $\mu_l(U) = \{\xi\in
\Gamma(U,\CO_U)\ |\ \xi^l = 1\}$. 
Now let $V = \Spec(\CC(X))$. Then by Hilbert 90,
\[
H_{\et}^1(\CC(X),\mu_l)   \simeq \CC(X)^{\times}/[\CC(X)^{\times}]^l,
\]
where $[\CC(X)^{\times}]^l = \{x^l\ |\ x\in \CC(X)^{\times}\}$.
There are exact sequences
\[
0 \to \CC^{\times}\to \CC(X)^{\times} \xrightarrow{d\log} \Omega^1_{\CC(X)/\CC},
\]
\[
0 \to \CC^{\times}\to [\CC(X)^{\times}]^l \xrightarrow{d\log} \Omega^1_{\CC(X)/\CC},
\]
where $ \Omega^1_{\CC(X)/\CC}$ are the K\"ahler differentials, which induces the short
exact sequences
\[
0 \to \CC^{\times}\to \CC(X)^{\times} \xrightarrow{d\log} \Gamma\big(
H^1(\CC(X),\ZZ(1))\big)\to 0
\]
and
\[
0 \to \CC^{\times}\to [\CC(X)^{\times}]^l \xrightarrow{d\log} l\cdot\Gamma\big(
H^1(\CC(X),\ZZ(1))\big)\to 0
\,.
\]
Thus going from the \'etale to Betti cohomology with finite coefficients
can be traced via the isomorphisms:
\begin{equation}\label{E36}
H_{\et}^1(\CC(X),\mu_l) \simeq \CC(X)^\times / [\CC(X)^{\times}]^l
\end{equation}
\[
\simeq
\frac{\Gamma\big(H^1(\CC(X),\ZZ(1))\big)}{l\cdot \Gamma\big(H^1(\CC(X),\ZZ(1))\big)}
\simeq \frac{H^1(\CC(X),\ZZ(1))}{l\cdot H^1(\CC(X),\ZZ(1))}
\]
\[
 \simeq
H^1\biggl(\CC(X),\frac{\ZZ(1)}{l\cdot \ZZ(1)}\biggr).
\]
Note that $d\log : \CH^1(\Spec(\CC(X),1)= \CC(X)^{\times} \to H^1(\CC(X),\ZZ(1))$,
and $\CH^1(\Spec(\CC(X),1)= \CC(X)^{\times} \to H_{\et}^1(\CC(X),\mu_l)$ from (\ref{E36})
are the respective cycle class maps to Betti cohomology (analytic topology) and
\'etale cohomology. 
In summary, we have a commutative diagram corresponding to a morphism
of sites from the \'etale to the analytic topologies.
\[
\begin{matrix}
\CC(X)^{\times}&\xrightarrow{d\log}&H^1(\CC(X),\ZZ(1))\\
&\\
\biggl\downarrow&&\biggr\downarrow\\
&\\
H_{\et}^1(\CC(X),\mu_l)&\xrightarrow{\sim}&
H^1\big(\CC(X),\frac{\ZZ(1)}{l\cdot \ZZ(1)}\big)
\end{matrix}
\]
where the isomorphism in the bottom row is from~\eqref{E36}.
Taking cup products, we have a similar diagram
\[
\begin{matrix}
\big(\CC(X)^{\times}\big)^{\otimes m}&\xrightarrow{(d\log)^{\otimes m}}&H^1(\CC(X),\ZZ(1))^{\otimes m}\\
&\\
\biggl\downarrow&&\biggr\downarrow\\
&\\
H_{\et}^1(\CC(X),\mu_l)^{\otimes m}&\xrightarrow{\sim}&
H^1\big(\CC(X),\frac{\ZZ(1)}{l\cdot \ZZ(1)}\big)^{\otimes m}\\
&\\
\cup\biggl\downarrow\quad&&\quad\biggr\downarrow\cup\\
&\\
H_{\et}^m(\CC(X),\mu_l^{\otimes m})&\xrightarrow{\sim}&
H^m\big(\CC(X),\frac{\ZZ(m)}{l\cdot \ZZ(m)}\big)^{\otimes m}
\end{matrix}
\]
where the isomorphism
\[
H^m_{\et}(\CC(X),\mu_l^{\otimes_{\ZZ}m})\xrightarrow{\sim}
H^m\biggl(\CC(X),\frac{\ZZ(m)}{l\cdot \ZZ(m)}\biggr)
\]
here (and in the previous diagram for $ m=1 $),
which really arises from the Leray spectral sequence associated
to a morphism of sites, 
can be deduced from \cite[Thm 3.12 on p.117]{Mi}. The Bloch-Kato/Milnor conjectures
(now theorems \cite{W}) tell us that for $l$ a non-zero integer,
the induced map
\[
\frac{\CH^m(\Spec(\CC(X)),m)}{l\cdot \CH^m(\Spec(\CC(X)),m)}
 \to H^m_{\et}(\Spec(\CC(X)),\mu_l^{\otimes_{\ZZ}m})
\]
is an isomorphism.
\footnote{This generalizes the Merkurjev-Suslin theorem, where $m=2$.}
In our situation, this translates to saying

\begin{theorem}\label{BK}
For $ m \ge 0 $, the  map
\begin{equation*}
\frac{\CH^{m}(\Spec(\CC(X)),m)}{l\cdot CH^{m}(\Spec(\CC(X)),m)} \to 
H^{m}\biggl(\CC(X),\frac{\ZZ(m)}{l\cdot \ZZ(m)}\biggr)
\end{equation*}
is an isomorphism for any integer $l\ne 0$.
\end{theorem}

We can now prove the following result.  Note that part~(i)
for $ i=1 $ is immediate from the short exact sequence~\eqref{expsequence},
and for $i=2$ follows from the Lefschetz $(1,1)$ theorem.

\begin{theorem}\label{T0001}
{\rm (i)} $H^i(\CC(X),\ZZ) $ is torsion-free for all $ i $. In particular,
the torsion subgroup of $H^i(X,\ZZ)$ is supported in codimension $1$.

\bigskip
{\rm (ii)} $\ker(d\log_m)$ in (\ref{E001}) is divisible.

\bigskip
{\rm (iii)} The groups
\[
\frac{H^{m}(\CC(X),\ZZ(m))}{\Gamma\big(H^{m}(\CC(X),\ZZ(m)\big)}
\quad\textup{\it and }\quad
\frac{\Gamma\big(H^m(\CC(X),\ZZ(m)\big)}{\IM(d\log_m)}
\]
are uniquely divisible.
\end{theorem}

\begin{proof} First observe that the map in Theorem~\ref{BK} is 
the composition
\[
\frac{\CH^{m}(\Spec(\CC(X)),m)}{l} \to
\frac{H^{m}(\CC(X),\ZZ(m))}{l} \to 
 H^{m}\biggl(\CC(X),\frac{\ZZ(m)}{l\cdot\ZZ(m)}\biggr).
\]
Notice that the short exact sequence
$ 0 \to \ZZ \xrightarrow{\times l} \ZZ \to \ZZ/l \ZZ\to 0$
induces the short exact sequence
$$
\frac{H^m(\CC(X),\ZZ(m))}{l}\hookrightarrow
H^m\biggl(\CC(X),\frac{\ZZ(m)}{l\cdot\ZZ(m)}\biggr) \twoheadrightarrow 
H^{m+1}(\CC(X),\ZZ(m))_{l-{\rm tor}}.
$$
By Theorem \ref{BK}, it follows that $H^{m+1}(\CC(X),\ZZ(m))_{l-{\rm tor}} = 0$,
thus proving part~(i). Next  observe that
 \[
 \Gamma \big(H^m(\CC(X),\ZZ(m)\big) = F^mH^m(\CC(X),\CC) \cap H^m(\CC(X),\ZZ(m))
\,,
  \]
  and hence
 \[
 \Gamma \big(H^m(\CC(X),\ZZ(m)\big) \cap l\cdot H^m(\CC(X),\ZZ(m)) = l\cdot
 \Gamma \big(H^m(\CC(X),\ZZ(m)\big).
 \]
Using also Theorem~\ref{BK}, we have the commutative diagram
\[
\xymatrix{
 \dfrac{\CH^m(\Spec(\CC(X)),m)}{l\cdot \CH^m(\Spec(\CC(X)),m)} \ar@{->>}[d] \ar[r]^-\simeq &
      H^m\biggl(\CC(X),\dfrac{\ZZ(m)}{l\cdot \ZZ(m)}\biggr)
\\
 \dfrac{\IM(d\log_m)}{l\cdot\IM(d\log_m)}\ar[d] & 
\\
 \dfrac{\Gamma\big(H^m(\CC(X),\ZZ(m))\big)}{l\cdot \Gamma\big(H^m(\CC(X),\ZZ(m))\big)}
 \ar@{^{(}->}[r] & \dfrac{H^m(\CC(X),\ZZ(m))}{l\cdot H^m(\CC(X),\ZZ(m))} \ar[uu]_-\simeq
\,,
}
 \]
where all maps must be isomorphisms.
Part~(ii) follows by applying the snake lemma
to multiplication by $ l \ne 0 $ on the short exact sequence
\begin{equation*}
 0 \to  \ker(d \log_m) \to \CH^m(\Spec(\CC(X)),m) \to \IM(d \log_m) \to 0 
\end{equation*}
as $ \IM(d \log_m) $ is torsion-free.  Using the obvious abbreviations,
part~(iii) follows similarly from
$ 0 \to \Gamma \to H^m \to H^m/ \Gamma \to 0  $
and
$  0 \to \IM \to \Gamma \to  \Gamma/ \IM \to 0   $.
\end{proof}

\begin{cor}
\[
\frac{\Gamma\big(H^m(\CC(X),\ZZ(m))\big)}{\IM(d\log_m)} =
\frac{\Gamma\big(H^m(\CC(X),\QQ(m))\big)}{\IM(d\log_m\otimes\QQ)}.
\]
\end{cor}

\begin{cor}\label{C99}
\[
\ker \biggl[\ul{AJ} : \frac{\CH^2_{\hom}(X,1)}{\CH^2_{\dec}(X,1)} \to
 J\biggl(\frac{H^2(X,\ZZ(2))}{H^2_{\alg}(X,\ZZ(2))}\biggr)\biggr]
 \]
 is uniquely divisible. 
\end{cor}

\begin{proof}
Apply Theorem~\ref{T0001}(iii) with $m=2$ to Corollary~\ref{C100}.
\end{proof}
 
Note that one can define $N^1\CH^r(X,m-1)$ with integral coefficients
analogous to~\eqref{E88}.
We assume  $ r\ge m$.  Since $H^{2r-m}(\CC(X),\ZZ(r))$ is
 torsion-free by Theorem \ref{T0001}(i),  diagram (\ref{E27}) becomes valid
 after passing to the generic point of $X$. After applying the snake lemma, we
 arrive at the fact that 
 \[
 \ker \biggl( \ul{AJ}: \frac{\CH^r(X,m-1)}{N^1\CH^r(X,m-1)}\to
 J\biggl(\frac{H^{2r-m}(X,\ZZ(m-1))}{N^1H^{2r-m}(X,\ZZ(m-1))}\biggr)\biggr)
 \]
 injects into
 \[
 \frac{\Gamma\big(H^{2r-m}(\CC(X),\ZZ(r))\big)}{\IM (\lim(\cl_{r,m}))}
\,.
 \]

 In the case $r=m$ we deduce
 \begin{cor}\label{C66}
 \[
 \ker\biggl( \ul{AJ}: \frac{\CH^m(X,m-1)}{N^1\CH^mX,m-1)}\to
 J\biggl(\frac{H^{m}(X,\ZZ(m-1))}{N^1H^{m}(X,\ZZ(m-1))}\biggr)\biggr)
 \]
 is torsion-free. Hence we have an injection of torsion subgroups
 \[
  \ul{AJ}: \biggl\{\frac{\CH^m(X,m-1)}{N^1\CH^mX,m-1)}\biggr\}_{\rm tor} \hookrightarrow
\biggl\{ J\biggl(\frac{H^{m}(X,\ZZ(m-1))}{N^1H^{m}(X,\ZZ(m-1))}\biggr)\biggr\}_{\rm tor}
\,.
\]
\end{cor}

\begin{rem}
Conjecture~\ref{C0001} would imply that the maps $ \ul{AJ} $
in Corollaries~\ref{C99} and~\ref{C66} are injective.
\end{rem}

\section{\ Decomposables}\label{SEC07}

Fix $ r,m \ge 1 $.
If $ r=m $, then according to Theorem~\ref{T01},
surjectivity of $\cl_{m,m}$ implies that
$\lambda_{m}$ is injective on $d_{m}(\widetilde{E}_{m}^{-m,0}) \subset 
\widetilde{E}_{m}^{0,-m+1}$.
Thus it makes sense to calculate $\IM(d_{m})$
in general. 

Let $j : Y\hookrightarrow X$
be the inclusion of a NCD $Y$ (with smooth components).
Then
$$
\frac{\widetilde{E}_{m}^{0,-m+1}}{\IM(d_{m})} = 
E_\infty^{0,-m+1}\otimes \QQ
= \frac{\CH^{r}(X,m-1;\QQ)}{j_{\ast}\CH^{r-1}(Y,m-1;\QQ)}\subset 
\CH^{r}(U,m-1;\QQ)
\,.
$$
Note that $\CH^{r-1}(Y,m-1;\QQ)$ can be calculated 
from the simplicial complex $Y^{[\bullet]}\to Y$. So there are residue maps
\[
\del^{Y}_{R} : \CH^{r-1}(Y,m-1;\QQ) \to \frac{\CH^{r-m}(Y^{[m]};\QQ)}{\Gy
\big(\CH^{r-m-1}(Y^{[m+1]};\QQ)\big)},
\]
\[
\del^{X\bs Y}_{R} : \CH^{r}(X\bs Y,m;\QQ) \to \frac{\CH^{r-m}(Y^{[m]};\QQ)}{\Gy
\big(\CH^{r-m-1}(Y^{[m+1]};\QQ)\big)}.
\]

\begin{prop}\label{P004}
For $ m \ge 1 $, 
$$
\IM(d_{m}) = \frac{j_{\ast}\CH^{r-1}(Y,m-1;\QQ)}{j_{\ast}(\ker\del^{Y}_{R})}
\,.
$$
\end{prop}

\begin{cor}  If $ m \ge1 $ and $\cl_{m,m}$ is surjective, then the Abel-Jacobi
map $ \lambda_m $ in~\eqref{E003} induces an injection
$$
\frac{j_{\ast}\CH^{m-1}(Y,m-1;\QQ)}{j_{\ast}(\ker\del^{Y}_{R})} \hookrightarrow J(Gr_{-m}^{W})
\,.
$$
\end{cor}

\begin{proof}[Proof of Proposition~\ref{P004}]
Consider the exact sequence
\begin{equation*}
\CH^{r}(X,m) \to \CH^{r}(X\bs Y,m) \to
 \CH^{r-1}(Y,m-1) \to  \CH^{r}(X,m-1)
\,.
\end{equation*}
In the weight filtered spectral sequence obtained from~\eqref{ss}
involved in computing $\CH^{r}(X\bs Y,m)$ we can restrict our
attention to those $ \Z_{0}^{i,j}(r) = z^{r+i}(Y^{[-i]},-j) $ with
$ i \le -1 $, which converges to 
$\CH^{r-1}(Y,m-1)$.  It also exists for $\CH^{r}(X,m-1)$ (using
the column where $ i=0 $).  We use indices to distinguish between the
various spectral sequences.

Note that $E_{1,X\bs Y}^{0,-m+1} = E^{0,-m+1}_{\infty,X} = \CH^{r}(X,m-1)$ and that 
for $\ell\geq 1$,
\begin{equation} \label{Es}
 E_{\ell,X\bs Y}^{-\ell,\ell-m} = E_{\ell,Y}^{-\ell,\ell-m} \twoheadrightarrow E_{\infty,Y}^{-\ell,\ell-m}
\end{equation}
since $d_{\ell}^{Y} = 0$ on $E_{\ell,Y}^{-\ell,\ell-m}$ as
its target is trivial.
Non-canonically, we have
\begin{alignat*}{1}
\CH^{r-1}(Y,m-1;\QQ) & \simeq \widetilde{E}^{-m,0}_{\infty,Y} 
   \oplus \ker(\del^{Y}_{R})
\\
\intertext{and}
\CH^{r}(X\bs Y,m;\QQ) &  \simeq \widetilde{E}^{-m,0}_{\infty,X\bs Y} 
   \oplus \ker(\del^{X\bs Y}_{R})
\,.
\end{alignat*}
Thus
\begin{equation*}
j_{\ast}\CH^{r-1}(Y,m-1;\QQ) \simeq 
\frac{\CH^{r-1}(Y,m-1;\QQ)}{\CH^{r}(X\bs Y,m;\QQ)}
\simeq \frac{\widetilde{E}^{-m,0}_{\infty,Y}}{\widetilde{E}^{-m,0}_{\infty,X\bs Y}}
\oplus \frac{ \ker(\del^{Y}_{R})}{ \ker(\del^{X\bs Y}_{R})}
\,,
\end{equation*}
with the last term isomorphic with $ j_*(\ker(\partial_R^Y)) $.
One has a commutative diagram
\begin{equation*}
\begin{matrix}
&&&&\widetilde{E}^{-m,0}_{m,X\bs Y}\\
&\\
&&&&\biggl\downarrow&\searrow^{d_{m}}\\
&\\
0&\to&\widetilde{E}^{-m,0}_{\infty,X\bs Y}&\to&
\widetilde{E}^{-m,0}_{\infty,Y}&\to&\widetilde{E}^{0,-m+1}_{m,X\setminus Y}
\end{matrix}
\end{equation*}
with surjective vertical map by~\eqref{Es}.
One sees the bottom row is exact by comparing the spectral
sequences $ E_{m,X\setminus Y}^{p,q} $ and $ E_{m,Y}^{p,q} $
at $ (p,q) = (-m,0) $.  Hence
$ j_{\ast}\CH^{r-1}(Y,m-1;\QQ) \simeq \IM(d_{m}) \oplus j_*\ker(\partial_R^Y)$.
\end{proof}

{\it  We now restrict to the case $r\geq m$}
because $\IM(d_{m}) = 0$ for $r<m$ as $\CH^{r-m}(Y^{[m]}) = 0$ (see (\ref{E666})).
Let $\xi \in CH^{r}(X,m-1)$.
Then $|\xi| \subset X\times \Delta^{m-1}$ is of codimension $r$,
and hence $W := \ol{\text{\rm Pr}_{1}(|\xi|)}$ is of codimension 
$(r-m) +1\geq 1$ in $X$. By Hironaka, there is a proper modification
diagram
\begin{equation}\label{E006}
\begin{matrix} \sigma^{-1}(W)&=:&\ol{Y}&{\buildrel \ol{j}
\over\hookrightarrow}&\ol{X}\\
&\\
&&\downarrow&&\wr\wr\downarrow \sigma\\
&\\
&&W&{\buildrel \ol{j}_{1}\over\hookrightarrow}&X
\end{matrix}
\end{equation}
where $\ol{Y}$ is a NCD. Let $U = X\bs W$.
We have a corresponding diagram
with $ \CH_{\ol{Y}}^r(\ol{X},m-1;\QQ) = \CH^{r-1}(\ol{Y},m-1;\QQ) $,
\begin{equation}\label{E007}
\begin{matrix}
\ol{j}_{\ast}\CH^{r}_{\ol{Y}}(\ol{X},m-1;\QQ)&\hookrightarrow&
\CH^{r}(\ol{X},m-1;\QQ)&\to&\CH^{r}(U,m-1;\QQ)\\
&\\
\sigma_{\ol{Y},\ast}\downarrow\quad&&\quad\downarrow\sigma_{\ast}&&||\\
&\\
\ol{j}_{1,\ast}\CH^{r}_{W}(X,m-1;\QQ)&\hookrightarrow&\CH^{r}(X,m-1;\QQ)&
\to&\CH^{r}(U,m-1;\QQ)
\end{matrix}
\end{equation}
where $ \xi $ is in $ \ol{j}_{1,\ast}\CH^{r}_{W}(X,m-1;\QQ) $.
Obviously $\sigma_{\ast}\circ\sigma^{\ast} =$ Identity implies
that $\sigma_{\ast}$ is onto. A diagram chase shows that 
$\sigma_{\ol{Y},\ast}$ is a surjection as well. As $\ol{X}$ is
obtained from $X$ by a sequence of blow-ups with non-singular centres,
it is clear
from the well-known motivic decomposition of $\ol{X}$ with respect $X$,
that the (higher) Chow group of $\ol{X}$ involves that
of $X$ and of smooth irreducible divisors. Hence it
follows that $\ker(\sigma_{\ast}) \subset \ol{j}_{\ast}(\ker 
\del^{\ol{Y}}_{R})$. We deduce the isomorphisms
\[
\frac{\CH^{r}(\ol{X},m-1;\QQ)}{\ol{j}_{\ast}\CH^{r}_{\ol{Y}}
(\ol{X},m-1;\QQ)} \simeq \frac{\CH^{r}(X,m-1;\QQ)}{\ol{j}_{1,\ast}
\CH^{r}_{W}(X,m-1;\QQ)}
\,,
\]
\[
\frac{\ol{j}_{\ast}\CH^{r}_{\ol{Y}}(\ol{X},m-1;\QQ)}{\ol{j}_{\ast}
(\ker \del_{R}^{\ol{Y}})} \simeq \frac{\ol{j}_{1,\ast}
\CH^{r}_{W}(X,m-1;\QQ)}{\sigma_{\ast}\circ 
\ol{j}_{\ast}(\ker \del_{R}^{\ol{Y}})}
\,.
\]
Let
$$
\Xi^{r,m}_{X} := \sum_{(\ol{X},\ol{Y})}\sigma_{\ast}
\circ\ol{j}_{\ast}(\ker \del^{\ol{Y}}_{R})
\,,
$$
where $ (\ol{X},\ol{Y}) $ runs over all proper modifications~\eqref{E006}
for all $ W \subset X $ of codimension at
least~1.  By construction we have
$$
\CH^{r}(X,m-1;\QQ) = \sum_{(\ol{X},\ol{Y})}\sigma_{\ast}
\circ\ol{j}_{\ast}\big(\CH^{r}_{\ol{Y}}
(\ol{X},m-1;\QQ)\big).
$$
Now recall the subspace $N^1\CH^r(X,m-1;\QQ)\subset \CH^r_{\hom}(X,m-1;\QQ)$
introduced in equation (\ref{E88}). It follows from the definitions
that 
\begin{equation*}
\Xi^{r,m}_{X} \subseteq N^1\CH^r(X,m-1;\QQ),
\end{equation*}
since, by~\eqref{E004},
$ \del_{R}^{\ol{Y}}(\xi) = 0 $ in $\widetilde{E}_1^{-m,0}$
implies $\b(\del_{R}^{\ol{Y}}(\xi)) = 0$ in $\Gamma\big(Gr_0^W\big)$.
The above inclusion is an  equality
if $r=m$, since in this case $\b$  is an isomorphism.
Let $\CH^{r}_{\dec}(X,m-1;\QQ) \subseteq \CH^{r}(X,m-1;\QQ)$
be the subspace generated by images of the form
$$
\CH^{a}(X,b;\QQ) \otimes \CH^{t}(X,s;\QQ) \to \CH^{r}(X,m-1;\QQ)
$$
under the product where
\[
a+s = r ,
\quad
b+t = m-1 ,
\quad
(a,b) \ne (0,0) ,
\quad
(s,t) \ne (0,0) .
\]

\begin{prop}\label{P005}
For $ r \ge m \ge 1 $ we have
$\CH^{r}_{\dec}(X,m-1;\QQ) \subseteq \Xi^{r,m}_{X}$.
\end{prop}

\begin{rem}
It is not clear if equality holds in the above proposition.
\end{rem}

\begin{proof}[Proof of Proposition~\ref{P005}]
Consider $\xi = \xi_{1}\otimes \xi_{2}\in 
\CH^{a}(X,b;\QQ)\otimes \CH^s(X,t;\QQ)$, where $b,\ t \leq m-1$. 
Note that $\CH^{i}(X,j)$ is supported in codimension $i-j$,
namely 
\[
\CH^{i}(X,j) = \sum_{{\rm cd}_XY \geq i-j}\IM\big(\CH^i_Y(X,j) \to \CH^i(X,j)\big).
\]
Also $a\leq b$ and $s\leq t$
$\Rightarrow r=a+s \leq b+t = m-1$, which is not the case as we are assuming
$r\geq m$.
Therefore we can assume say $s>t$.  Next, if $t=m-1$, then
$b=0$, and thus  $a> 0 = b$ implies $\xi_{1}$ is supported
on a divisor in $X$. This scenario can be handled
in the same way as in the case where we assume
that $s>t$ and $t<m-1$. Namely, since $ s>t $, we can assume that
$\xi_{2}$ is supported on some $Y\subset X$ of codimension $1\leq s-t$,
and by the surjectivity of $\sigma_{\ol{Y},\ast}$ in diagram (\ref{E007}),
we can assume without loss of generality that $Y\subset X$ is a NCD.
Thus for some $\xi_{2}^{Y}\in \CH^{s-1}(Y,t;\QQ)$,
$\xi_{2}^{Y}\mapsto \xi_{2}$. Since $X$ is smooth, we have
the product (\cite{B}):
$$
\CH^{a}(X,b) \times \CH^{s-1}(Y,t;\QQ) \to \CH^{r-1}(Y,m-1;\QQ),
$$
which defines $\xi_{1}\cap \xi_{2}^{Y} \in \CH^{r-1}(Y,m-1;\QQ)$.
But since $\xi_{1}$ can be assumed in general position with
respect to $Y$, and together with $t < m-1$, we
have $\del^{Y}_{R}\big(\xi_{1}\cap \xi_{2}^{Y}\big) = 0$. 
\end{proof}

\section{ Noether-Lefschetz for a family of surfaces}\label{SEC013}

Let $X/\CC$ be a smooth projective surface, $Y = \bigcup_{j=1}^n Y_i\subset X$ a NCD
(with smooth $Y_i$'s)  with open complement $U := X\bs Y$. The work
of Deligne (\cite{De}, Cor. 3.2.13 and 3.2.14) implies that 
$F^2H^2(U,\CC) = H^0(X,\Omega^2_X\langle Y\rangle)$, where
$\Omega^2_X\langle Y\rangle$ is the sheaf of rational $2$-forms
on $X$, regular on $U$ with logarithmic poles along $Y$. In particular
if we denote by $H(U)$,  the space of regular algebraic $2$-forms on $U$
with $\QQ(2)$-periods, then
\begin{equation}\label{E555}
\Gamma H^2(U,\QQ(2)) := H^2(U,\QQ(2)) \cap F^2H^2(U,\CC) \subset H(U),
\end{equation}
and where we allow for the possibility that the left hand side of (\ref{E555}) is
non-zero. 
Let $Z(Y)$ be the singular set of $Y$. For each $Y_j$ we choose
distinct points $\{P_j,Q_j\} \subset Y_j\bs \{Y_j\cap Z(Y)\}$. Now
let's modify $X$ by  blowing it up along $\{P_1,Q_1,...,P_n,Q_n\}$
and call this $X'$. So in particular  the strict transform $Y'$ of $Y$ is a copy of 
$Y$ itself.  On  $U' := X'\bs Y'$ we have for each $j= 1,...,n$
an interesting real $2$-cycle $\gamma_j$ obtained as follows: 
Take the complement of a small disc in the blowup 
of $P_j$, the complement of a small disc in the blowup of $Q_j$, and a small tube in $X\bs Y$
along a path in $Y_j$ from $P_j$ to $Q_j$ that so that the end circles meet the circles of the two 
previous parts.  Then by a standard residue argument,
integrating an element $\omega \in H(U')$  against 
$\gamma_j$ gives us essentially $2\pi\I$ times the integral along the 
part in $Y_j$  from $P_j$ to $Q_j$ of the residue of 
$\omega$ along $Y$, where we observe that $\omega$ restricts to zero on the other two 
parts of $\gamma_j$, viz., there are no non-zero holomorphic $2$-forms living
on any  subset of $\PP^1$.

\begin{figure}[htbp]
\begin{center}
\input{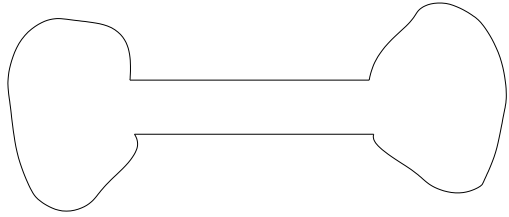}
\caption{The cycle $ \gamma_j $}
\label{dogbonecycle}
\end{center}
\end{figure}

\comment
\centerline{Picture of $\gamma_j \subset U'$}
\[
\begin{matrix}\bigcap_{--------------}\bigcap\\
\big|\big(\big)\hskip.53in\big(\big)\hskip.53in\big(\big)\big|\\
\bigcup^{--------------}\bigcup\\
&\\
\cap\hskip1.35in\cap\\
&\\
\ \PP^1\hskip1.3in\PP^1
\end{matrix}
\]
\endcomment
Note that this integral is determined by the finite dimensional $\QQ$-vector 
space of residues on $Y$.  So pick $P_j$ and $Q_j$ sufficiently general in $Y_j$ 
so that we can never 
end up in $\QQ(2)$ with this integral except with the trivial residue. Then doing this for all 
$j=1,...,n$, we arrive at $U'$ for which 
$\Gamma H^2(U',\QQ(2)) = 0$, using the fact that $H^0(X',\Omega^2_{X'})
\cap H^2(X',\QQ(2)) = 0$. We deduce:

\begin{theorem} Let 
\[
t \in B:= \biggl\{(P_1,Q_1,...,P_n,Q_n) \in \prod_{j=1}^n Y_j^2 \left| \begin{matrix}
P_j\ne Q_j\\ \{P_j,Q_j\}\in Y_j\bs \{Y_j\cap  Z(Y)\}\\
\ j=1,...,n\end{matrix}\right\},
\]
with corresponding family $\{U'_t\}_{t\in B}$. Then $\Gamma H^2(U_t',\QQ(2)) = 0$ for a very general point
$t$ in $B$. 
\end{theorem}

\section{Appendix: The Abel-Jacobi map revisited}\label{SEC05}

The purpose of this appendix is to establish the commutativity of diagram~\eqref{E003} above.
Due to its  technical nature,
the reader with pressing obligations
can easily skip this without losing sight  of the main results
of this paper. We first 
digress by describing the Abel-Jacobi map
\[
\lambda_k : \widetilde{E}_k^{-m+k,-k+1} \to J\big(Gr^W_{-k}\big),
\]
where $Gr^W_{-k}$ is {described} in~\eqref{111}.
In the case $k=1$, viz.,
\[
\lambda_1 : \widetilde{E}_1^{-m+1,0} \to J\big(Gr^W_{-1}\big),
\]
this is  induced by the classical Abel-Jacobi map
\[
\CH^{r-m+1}_{\hom}(Y^{[m-1]};\QQ) \xrightarrow{\xi \mapsto\int_{\del^{-1}\xi}(-)}
 J\big(H^{2r-2m+1}(Y^{[m-1]},\QQ(r-m+1))\big).
\]
The Abel-Jacobi map
\[
\Phi_k: \widetilde{E}_2^{-m+k,-k+1} \to J\big(Gr^W_{-k}\big),
\]
is defined using the formula in \cite{KLM}. By degeneration of the
mixed Hodge complex spectral sequence at $E_2$ (Deligne), and a map of spectral sequences
from Chow groups to Hodge cohomology, together with functoriality of
the Abel-Jacobi map, the map $\Phi_k$ induces
$\lambda_k$ for all $k\geq 2$. So we need only describe
the map $\Phi_k$ explicitly. 
Let $\square^m:=({\PP^1} \backslash
\{1\})^m $ with coordinates $z_i$ and $2^{m}$ codimension one faces
obtained by setting $z_i=0,\infty$, and boundary maps
$\partial = \sum (-1)^{i-1}(\partial_{i}^{0}-\partial_{i}^{\infty})$,
where $\partial_{i}^{0},\ \partial_{i}^{\infty}$ denote
the restriction maps to the faces $z_{i}=0,\ z_{i}=\infty$ 
respectively.  Here we adopt the notation
in \cite{KLM} adapted to the cubical description of $\CH^{r}(X,m;\QQ)$,
with cycles lying in $z^r(X\times \square^m;\QQ)$, in general position
with respect to the $2^m$ faces of $\square^m$ as well as 
the real part $[-\infty,0]^m \subset \square^m$. 

\smallskip
Recall the Tate twist  $\QQ(r)$,
let $X/\CC$ be smooth projective with
$d = \dim X$, and put $\Dd^k_X :=$ sheaf of currents that act on compactly
supported $\CC$-valued $C^{\infty}$ forms of degree $2d-k$.  Note that 
\[
\Dd_X^{\bullet} = \bigoplus_{p+q=\bullet}\Dd_X^{p,q},
\]
where $\Dd^{p,q}_X$ acts on corresponding $(d-p,d-q)$ forms. Let
$\C^k_X(\QQ(r))$ be the sheaf of Borel-Moore chains of real codimension
$k$ in $X$ with $\QQ(r)$ coefficients. One has an inclusion $\C^k_X(\QQ(r)) \subset \Dd_X^k$.
Now put
\[
{\M}^{\bullet}_{\Dd} = {\rm Cone}\big\{\C_X^{\bullet}(X,
\QQ(r))) \bigoplus 
F^r\Dd_X^{\bullet}(X)\to 
\Dd_X^{\bullet}(X)\big\}[-1].
\]
The cohomology of this complex at $\bullet = k$
is precisely the Deligne cohomology $H^k_{\Dd}(X,{\QQ}(r))$
(see \cite{KLM}).

We recall the cycle class map
\[
\Psi_{r,m} : \CH^r(X,m;\QQ) \to H^{2r-m}_{\Dd}(X,\QQ(r)),
\]
in terms of the cubical description of $\CH^r(X,m)$.  Note that
for $m\geq 1$, 
\[
H^{2r-m}_{\Dd}(X,\QQ(r) \simeq J\big(H^{2r-m-1}(X,\QQ(r))\big),
\]
and under this identification, $\Psi_{r,m}$ is the Abel-Jacobi map.
On $\square^m$
we introduce the currents
\[
\Omega_{\square^m} := \bigwedge_{1}^{m}d\log z_{j}
\]
\[
T_{\square^m} :=  T_{z_1}\cap\cdots\cap T_{z_m} = \int_{[-\infty,0]^m}(-) =: \delta_{[-\infty,0]^m},
\]
where $T_{z_j}$ is integration on $\big\{(z_1,...,z_m)\in \square^m\ \big|\ z_j\in [-\infty,0]\big\}$.
\[
R_{\square^m} := \log z_1d\log z_2\wedge\cdots\wedge d\log z_m
- (2\pi\I) \log z_2d\log z_3\wedge\cdots\wedge d\log z_m\cdot {T_{z_1}}
\]
\[
+\cdots + (-1)^{m-1}(2\pi\I)^{m-1}\log z_m\cdot \{T_{z_1}\cap \cdots\cap T_{z_{m-1}}\},
\]
where $\log$ has the principle branch.
One has
\[
d R_{\square^m} = \Omega_{\square^m} - (2\pi\I)^mT_{\square^m} - 2\pi\I R_{\del\square^m}.
\]
We consider a cycle 
$\xi\in z^{r}(X\times \square^{m})$ in general position. One considers
the projections $\pi_{1}:|\xi| \to X,\ \pi_{2} : |\xi|\to \square^{m}$.
We put 
$$
R_{\xi} = \pi_{1,\ast}\circ \pi_{2}^{\ast}R_{\square^m},\quad
\Omega_{\xi} = \pi_{1,\ast}\circ \pi_{2}^{\ast}\Omega_{\square^m},\quad
T_{\xi} = \pi_{1,\ast}\circ \pi_{2}^{\ast}T_{\square^m}.
$$
Note that when $m=0$, $R_{\xi}  = 0$ (classical case!). Correspondingly
\[
d R_{\xi} = \Omega_{\xi} - (2\pi\I)^mT_{\xi} - 2\pi\I R_{\del\xi}
\,.
\]
Up to a normalizing constant, the cycle class map  $\Psi_{r,m}$ is induced by
\[
\xi \mapsto \big((2\pi\I)^mT_{\xi},\Omega_{\xi},R_{\xi}\big).
\]
Recall $d=\dim X$. The Abel-Jacobi map 
\begin{equation*}
\Phi_{r,m} : \CH^r_{\hom}(X,m;\QQ) \to
\frac{F^{d-r+1}H^{2d-2r+m+1}(X,\CC)^{\vee}}{H_{2d-2r+m+1}(X,\QQ(d-r))}
\,,
\end{equation*}
where we described $ J\big(H^{2r-m-1}(X,\QQ(r))\big) $ using the Carlson isomorphism,
is defined as follows. If $\xi\in \CH^r(X,m;\QQ)$, then by a moving lemma
\cite[Lemma~8.14]{K-L}, we can assume that $\xi$ is in general position
with respect to the real cube $[-\infty,0]^m\subset \square^m$. 
Furthermore, $m > 0 $ implies $ T_{\xi} = dT_{\zeta}$,
i.e., $\pm\del\zeta = \xi\cap \{X\times [-\infty,0]^m\}$, $\Omega_{\xi} =  -dS$, for
some $S\in F^r\Dd_X^{2r-m-1}(X)$, so up to a normalizing constant and for
$\omega\in F^{d-r+1}H^{2d-2r+m+1}(X,\CC)$, we have
\[
\Phi_{r,m}(\xi)(\omega) = S (\omega)+  (-2\pi\I)^m\int_{\zeta} \omega + R_{\xi}(\omega) =   
(-2\pi\I)^m\int_{\zeta}\omega  + R_{\xi} (\omega),
\]
where the latter equality stems from Hodge type considerations.

\smallskip
Now referring to~\eqref{E003}, we fulfill a promise made earlier, viz.,

\begin{theorem} \label{diagrams-commute}
The diagram~\eqref{E003} is commutative. Specifically,
\[
h_k\circ \a_k = \lambda_k\circ d_k.
\]
\end{theorem}

\begin{proof}
We have to unravel the definitions.
We use the simplicial complex $Y^{[\bullet]}\to Y\hookrightarrow X$
as a way of describing $W_jH^{2r-m}(U,\QQ(r))$. Let 
\[
\K_{\QQ}^{2r-2m+i}(Y^{[m-i]}) = \C^{2r-2m+i}(Y^{[m-i]},\QQ(r-m+i))
\]
\[
\K_{\CC}^{2r-2m+i}(Y^{[m-i]}) = \Dd^{2r-2m+i}(Y^{[m-i]})
\]
\[
\D = d \pm \Gy
\]
A class $\xi\in W_0H^{2r-m}(U,\QQ(r))$ is represented by a
$\D$-closed $(m+1)$-tuple 
\[
\xi = (\xi_0,\xi_1,...,\xi_m) \in \bigoplus_{i=0}^m\K_{\QQ}^{2r-2m+i}(Y^{[m-i]}).
\]
With $W_j := W_jH^{2r-m}(U,\QQ(r))$, consider the short
exact sequence
\[
0 \to W_{-1} \to W_0 \xrightarrow{\xi\mapsto\xi_0} Gr_0^W \to 0.
\]
Let us first describe $h_1\circ \a_1 : \widetilde{E}_1^{-m,0} \to J\big(Gr^W_{-1}\big)$.
Let $\gamma\in \widetilde{E}_1^{-m,0}$ and $\xi_0 = \b(\gamma)$.
In this case $\xi_0\in \Gamma\big(Gr_0^W\big)$ is in the image of $\xi\in W_0$. Likewise
$\xi_0$ is in the image of some $\xi^{\CC}\in F^0W_{0,\CC}$. The difference
$\xi-\xi^{\CC}\in W_{-1,\CC}$ maps to a class in $J\big(W_{-1}\big)$ which 
defines the Abel-Jacobi image of $\gamma$ in  $J\big(W_{-1}\big)$. We can
assume that $\xi-\xi^{\CC}$ is represented by the $\D$-closed $m$-tuple:
\[
 (\xi_1-\xi_1^{\CC},...,\xi_m-\xi_m^{\CC}) \in \bigoplus_{i=1}^m\K_{\CC}^{2r-2m+i}(Y^{[m-i]}).
 \]
with $\Gy(\xi_0) = \del\xi_1$. The corresponding value $h_1\circ \a_1(\gamma)
\in J\big(Gr_{-1}^W\big)$ is
given by the Abel-Jacobi membrane integral $\int_{\xi_1}(-)$, as the Hodge
contribution given by $\xi_1^{\CC}$ is trivial for Hodge type reasons.  This is easily seen to
be precisely $\lambda_1\circ d_1(\gamma)$, where we
recall that $d_1 = \Gy$. Thus $\lambda_1\circ d_1 = h_1\circ \a_1$. So now
suppose that $d_1(\gamma) = 0\in \widetilde{E}_1^{-m+1,0}$, 
i.e. $\gamma\in \widetilde{E}_2^{-m,0}$. This means that 
$\Gy(\gamma) = \del\zeta_1$, where $\zeta_1\in z^{r-m+1}(Y^{[m-1]},1;\QQ)$.
Then $d_2(\gamma) = \Gy(\zeta_1) \in \widetilde{E}_2^{-m+2,-1}$.
(We comment in passing that
using the aforementioned moving lemma in \cite{K-L}, we can assume that 
$\zeta_{1,\RR} := \zeta \cap Y^{[m-1]}\times [-\infty,0]$ is a proper
intersection, and hence that $\del \zeta_{1,\RR} = \Gy(\xi_0)$.)
Since $d_1(\gamma) = \Gy(\gamma)$ is a coboundary, it follows that $h_1\circ\a_1(\gamma) = 
\lambda_1\circ d_1(\gamma) = 0$,
and hence after removing classes in $W_{-1}+ F^0W_{-1,\CC}$, 
[specifically, $\del(\zeta_{1,\RR} - \xi_1) = \Gy(\xi_0)-\Gy(\xi_0) =0$, and using
$W_{-1}\twoheadrightarrow Gr_{-1}^W$, $d(\Omega_{\zeta_1}- \xi_1^{\CC})
= \Gy(\xi_0^{\CC}) - \Gy(\xi_0^{\CC}) = 0$ and using $F^0W_{-1,\CC} 
\twoheadrightarrow F^0Gr_{-1,\CC}^W$], we can assume
that $\xi-\xi^{\CC}$ is represented by the $\D$-closed $m$-tuple
\[
 (2\pi\I T_{\zeta_{1}} - \Omega_{\zeta_1},\xi_2-\xi_2^{\CC},...,\xi_m-\xi_m^{\CC}) \in \bigoplus_{i=1}^m\K_{\CC}^{2r-2m+i}(Y^{m-i]}),
 \]
where
\[
d\xi_2 = (2\pi\I)T_{\Gy(\zeta_{1})},\quad d\xi_2^{\CC} = \Omega_{\Gy(\zeta_1)}.
\]
Note that 
\[
dR_{\zeta_1} = \Omega_{\zeta_1} - (2\pi\I)T_{\zeta_{1}} -
(2\pi\I)R_{\del \zeta_1}  = \Omega_{\zeta_1} - (2\pi\I)T_{\zeta_{1}},
\]
as $R_{\del \zeta_1} = 0$ ($\del \zeta_1$ representing  the classical case!).
But $\del \Gy(\zeta_1) = \Gy(\del \zeta_1) = \Gy^2(\gamma) = 0$. Hence
working modulo the coboundary
\[
\D R_{\zeta_1} = dR_{\zeta_1} \pm  R_{\Gy(\zeta_1)},
\]
we can assume 
that $\xi-\xi^{\CC}$ is represented by the $\D$-closed $(m-1)$-tuple
\[
(\xi_2-\xi_2^{\CC} + R_{\Gy(\zeta_1)},\xi_3-\xi_3^{\CC},\dots,\xi_m-\xi_m^{\CC}) \in \bigoplus_{i=2}^m\K_{\CC}^{2r-2m+i}(Y^{m-i]})
\,.
 \]
So modulo $\D$-coboundary, $h_2\circ \a_2(\gamma)$ is
represented by the $d$-closed current
$ \xi_2 - \xi_2^{\CC} + R_{\Gy(\zeta_1)} $,
which is precisely $\lambda_2\circ d_2(\gamma)$.
Hence $h_2\circ\a_2=\lambda_2\circ d_2$. The general case
$h_k\circ\a_k=\lambda_k\circ d_k$ proceeds in a similar fashion.
\end{proof}

\end{document}